\documentclass[a4paper,11pt]{amsart}

\usepackage{amsmath}
\usepackage{amsthm}
\usepackage{amssymb}
\usepackage{amsbsy}
\usepackage{amsfonts}
\usepackage{amstext}
\usepackage{amscd}
\usepackage{tikz}
\usepackage{graphicx}
\usepackage{verbatim}
\usepackage{bm}
\usepackage{commath,mathtools,setspace}
\usepackage{paralist}
\usepackage{extarrows}
\usepackage{color}
\usepackage{stmaryrd}
\usepackage{mathrsfs}
\usepackage{stackengine}
\usepackage{comment}
\usepackage{stmaryrd}  

\usepackage[T1]{fontenc}
\usepackage{graphicx}
\usepackage{blindtext} 
\usepackage{geometry}
\usepackage[colorlinks=true,linkcolor=blue,citecolor=blue]{hyperref}
\usepackage[colorinlistoftodos,prependcaption,textsize=small]{todonotes}

\definecolor{red}{RGB}{255,0,0}
\definecolor{green}{RGB}{0,100,0}
\definecolor{blue}{RGB}{0,0,255}

\numberwithin{equation}{section}

\newtheorem{theorem}{Theorem}[section]
\newtheorem{lemma}[theorem]{Lemma}

\newtheorem{corollary}[theorem]{Corollary}
\newtheorem{proposition}[theorem]{Proposition}
\newtheorem{notation}[theorem]{Notation}

\newtheorem{remark}[theorem]{Remark}
\newtheorem{definition}[theorem]{Definition}
\newtheorem{example}[theorem]{Example}

\newcommand{\weak}{\xrightarrow{w}}
\newcommand{\vague}{\xrightarrow{v}}
\newcommand{\C}{\mathbb{C}}

\newcommand{\N}{\mathbb{N}}
\newcommand{\E}{\mathbb{E}}
\renewcommand{\P}{\mathbb{P}}
\newcommand{\R}{\mathbb{R}}

\newcommand{\dx}{{\rm d} x}
\newcommand{\dt}{{\rm d} t}

\newcommand{\polplus}{\mathcal P_n(\mathbb R_{\geq 0})}
\newcommand{\pols}{\mathcal P}

\newcommand{\kk}{\mathbb{K}}

\newcommand{\MM}{\mathcal{M}}

\def \X {{\mathbb X}}
\def \Y {{\mathbb Y}}
\def \A {{\mathbb A}}
\def \B {{\mathbb B}}

\def \M {{\mathbb M}}

\newcommand{\falling}[2]{\left(#1\right)_{#2}}
\newmuskip\pFqmuskip

\newcommand*\pFqN[6][8]{%
  \begingroup 
  \pFqmuskip=#1mu\relax
  \mathcode`\,=\string"8000
  \begingroup\lccode`\~=`\,
  \lowercase{\endgroup\let~}\pFqcomma
  {}_{#2}F_{#3}{\left(\genfrac..{0pt}{}{#4}{#5};#6\right)}%
  \endgroup
}
\newcommand{\pFqcomma}{\mskip\pFqmuskip}

\newcommand{\mset}[1]{%
  \{\!\!\{ #1 \}\!\!\}%
}

\newcommand*\HGP[3]{%
 \mathcal{H}_{#1}{\left[\genfrac..{0pt}{1}{#2}{#3}\right]}%
}

\newcommand{\meas}[1]{\mu\! \left\llbracket #1 \right\rrbracket}

\newcommand{\coef}[2]{\widetilde{\mathsf{e}}_{#1}^{(#2)}} 

\newcommand{\strans}[2]{S_{#1}^{(#2)}}   

\usepackage{stackengine,scalerel}

\makeatletter
\@namedef{subjclassname@2020}{%
	\textup{2020} Mathematics Subject Classification}
\makeatother


\begin{document}

\title{Finite free perpetuities}
\author{Julia Le Bihan}
\email{julia.le\_bihan.stud@pw.edu.pl}

\author[B. Ko\l{}odziejek]{Bartosz Ko\l{}odziejek}
\email{bartosz.kolodziejek@pw.edu.pl}
\address{Faculty of Mathematics and Information Sciences, Warsaw University of Technology, Koszykowa 75, \mbox{00-662} Warsaw, Poland}

\thanks{This research was funded in part by National Science Centre, Poland, 2023/51/B/ST1/01535.
\\ For the purpose of Open Access, the authors have applied a CC-BY public copyright licence to any Author Accepted Manuscript (AAM) version arising from this submission.}

\begin{abstract}
We introduce and study finite free perpetuities, defined as monic polynomial solutions of degree \(n\) to the affine fixed-point equation 
\[
p(z)
=
\E\!\left[
A^{n}\,p\!\left(\frac{z-B}{A}\right)\mathbf{1}_{\{A\neq0\}}
\right]
+
\E\!\left[
(z-B)^n\mathbf{1}_{\{A=0\}}
\right],
\]
where \(A\) and \(B\) are complex-valued random variables with finite moments up to order \(n\).
Equivalently, if \(p(z)=\E[(z-X)^n]\), then \(p\) encodes a truncated moment version of the classical perpetuity equation \(X\stackrel{d}{=}AX+B\) with $X$ and $(A,B)$ independent.
This places finite free perpetuities between classical perpetuities and free-probabilistic fixed-point laws.

We prove existence and uniqueness under weak conditions, and we identify a broad class of admissible pairs $(A,B)$ for which the resulting polynomial has only real, nonnegative zeros.
Our approach uses finite free additive and multiplicative convolutions together with a probabilistic representation via the \(U\)-transform.

As a motivating example, we exhibit an explicit family of finite free perpetuities expressed in terms of Jacobi polynomials and show that their empirical root distributions converge to a free-beta-prime law.
More generally, for admissible sequences of parameters, we prove weak convergence of the empirical root distributions of finite free perpetuities to the law of a free perpetuity characterized by the corresponding free fixed-point equation.
This yields a finite-degree polynomial model approximating free perpetuities and clarifies the connection between classical affine recursions, finite free convolutions, and free probability.
\end{abstract}

\maketitle

\section{Introduction}
The main object of this paper is a polynomial fixed-point equation of the form
\begin{equation}\label{eq:affine_fixed_point_intro}
p(z)
=
\E\!\left[
A^{n}\,p\!\left(\frac{z-B}{A}\right)\mathbf{1}_{\{A\neq0\}}
\right]
+
\E\!\left[
(z-B)^n\mathbf{1}_{\{A=0\}}
\right],
\end{equation}
where $p$ is an unknown polynomial and \(A\), \(B\) are complex-valued random variables with finite moments up to order \(n\).
 More generally, this can be viewed as a special case of the general scheme
\[
p(z)=\E\big[\psi_z(p)\big],
\]
where, for each \(z\in\C\), \(\psi_z\) is a random operator on the space of monic polynomials of degree \(n\).

A solution \(p\) of \eqref{eq:affine_fixed_point_intro} will be called a finite free perpetuity of order \(n\).

The terminology is motivated by the following moment interpretation.
Every monic polynomial \(p\) of degree \(n\) can be represented in the form
\[
p(z)=\E[(z-X)^n]
\]
for some complex-valued random variable \(X\).
Comparing coefficients shows that \eqref{eq:affine_fixed_point_intro} is equivalent to
\[
\E[X^k]=\E[(AX+B)^k],\qquad k=1,\ldots,n,
\]
where the joint law of \((X,A,B)\) satisfies the factorization condition 
\[
\E[X^iA^iB^j]=\E[X^i]\E[A^iB^j],
\qquad i,j\ge 0,\quad i+j\le n,
\]
which we call affine $n$-independence of $X$ and $(A,B)$. 

This is closely related to the classical perpetuity equation 
\[
X\stackrel{d}{=}AX+B,
\qquad X\text{ and }(A,B) \mbox{ are independent},
\]
which has been extensively studied; see, for example, the monograph \cite{BurBook16}.
In particular, every classical perpetuity with finite moments up to order \(n\) canonically defines a finite free perpetuity of order \(n\). Thus finite free perpetuities should be viewed as truncated analogues of classical perpetuities: one keeps only the first \(n\) moments of the law of \(X\), and replaces full independence by factorization of mixed moments up to order \(n\).

A second source of motivation comes from finite free convolutions, introduced by Marcus, Spielman and Srivastava in \cite{MSS} and subsequently developed, among others, in \cite{AP18,AGP23,AFPU24,Fujie25}.
When \(A\) and \(B\) are independent (this condition can again be  weakened to factorization of joint moments up to order $n$), \eqref{eq:affine_fixed_point_intro} can be rewritten as
\begin{align}\label{eq:simpleintro}
p=(p_A\boxtimes_n p)\boxplus_n p_B,
\end{align}
where
\[
p_A(z)=\E[(z-A)^n],
\qquad
p_B(z)=\E[(z-B)^n],
\]
and \(\boxtimes_n\), \(\boxplus_n\) denote the finite free multiplicative and additive convolutions.
From this perspective, finite free perpetuities form an intermediate model between classical affine recursions and their free-probabilistic counterparts.

Finite free convolutions are often realized through matrix models, or equivalently through expected characteristic polynomials of random matrix operations \cite{MSS,Mar21}. Suppose, for instance, that $p_A$ and $p_B$ have real roots, with the roots of $p_A$ nonnegative, and let $\mathbf A$ and $\mathbf B$ be symmetric $n\times n$ matrices such that
\[
p_A(z)=\det(zI_n-\mathbf A),
\qquad
p_B(z)=\det(zI_n-\mathbf B).
\]
If $\mathbf X$ is a symmetric matrix with characteristic polynomial
\[
p(z)=\det(zI_n-\mathbf X),
\]
then the right-hand side of \eqref{eq:simpleintro} can be represented as the expected characteristic polynomial of 
\[
\mathbf A^{1/2}V_1\mathbf XV_1^\top \mathbf A^{1/2}
+
V_2\mathbf BV_2^\top,
\]
where $V_1$ and $V_2$ are independent Haar-distributed orthogonal matrices. Thus, in this matrix-model picture, \eqref{eq:simpleintro} corresponds to the characteristic-polynomial fixed-point relation
\[
\det(zI_n-\mathbf X)
=
\E\left[\det\left(
zI_n-
\mathbf A^{1/2}V_1\mathbf XV_1^\top \mathbf A^{1/2}-
V_2\mathbf BV_2^\top
\right)\right],
\]
where the expectation is taken with respect to $V_1$ and $V_2$. 
In the present paper, however, we work directly at the polynomial level, using the $U$-transform approach introduced in \cite{Mar21}.

Our main results go in two directions.
First, for each fixed degree $n$, we establish general existence and uniqueness criteria for \eqref{eq:affine_fixed_point_intro}. We also  identify a broad class of pairs $(A,B)$ for which the corresponding polynomial $p$ has only real, nonnegative roots.

Second, we study the large-degree limit. For admissible sequences of parameters $(A_n,B_n)_n$, we prove that the empirical root distributions of the associated finite free perpetuities converge weakly to a probability measure $\mu_{\X}$. This limiting measure is characterized as the distribution of the unique solution of a free perpetuity fixed-point equation of the form
\[
\X \stackrel{d}{=} \A\, \X \,\A^\ast + \B,
\qquad \X\mbox{ and }
(\A,\B)\ \text{are $*$-free},
\]
as studied in \cite{FreePerp}. We also note that matrix analogues of perpetuities were recently studied in \cite{MatrixPerp}, where the large-dimensional limits of empirical spectral distributions were also related to laws of free perpetuities.

\subsection*{Structure of the paper}

Section~\ref{sec:preli} collects the notation on polynomials, empirical root distributions, and finite free convolutions, and recalls the approximation of free additive and multiplicative convolutions by their finite counterparts.

Section~\ref{sec:ex} presents a motivating family of examples based on Jacobi polynomials.
There we solve an explicit polynomial fixed-point equation and identify the weak limit of the empirical root distribution with a free perpetuity studied in \cite{FreePerp}.

Next, in Section \ref{sec:U} we introduce normalized derivatives and the \(U\)-transform, which provide a convenient probabilistic representation of finite free convolutions and allow us to reformulate polynomial fixed-point equations in moment form.

In Section \ref{sec:ffp} we define finite free perpetuities in general, prove existence and uniqueness under natural assumptions, and establish criteria ensuring that all roots are real and nonnegative.

Finally, we study the large-degree limit in Section \ref{sec:approx}.
For admissible sequences of parameters, we prove weak convergence of the empirical root distributions of finite free perpetuities to the law of a free perpetuity, thereby making precise the connection between the finite and free theories.

For completeness, Appendix~\ref{app} recalls the standard construction of complex-valued classical perpetuities and explains how it reduces to the usual real vector-valued setting.

\section{Preliminaries}\label{sec:preli}

\subsection{Polynomials - notation and properties}

Unless otherwise stated, all polynomials are assumed to belong to the polynomial ring $\C[z]$, and their roots are taken in $\C$.

\begin{notation}
Let $\kk$ be a Borel subset of $\C$. 

\begin{itemize}
\item We denote by $\MM(\kk)$ the set of Borel probability measures supported in $\kk$. 

\item For a measure $\mu\in \MM(\kk)$ we denote by $m_n(\mu)$ the value of the integral $\int_{\kk}x^n\mu(\dx)$ (if it exists).

\item We denote by $\pols_n(\kk)$ the set of all monic polynomials of degree $n\in \N$ with all roots in $\kk$. 

\item We denote by $\tilde{\pols}_{\le n}(\kk)$ the set consisting of all polynomials of degree less or equal to $n$ with all roots in $\kk$ and the zero polynomial.
\end{itemize}

\end{notation}

\begin{definition}
For a polynomial $p$ of degree $n$ with roots $\lambda_1(p), \dots, \lambda_n(p)$ (counting multiplicities), its empirical root distribution is the measure
\[ \meas{p} = \frac{1}{n}\sum_{i=1}^n \delta_{\lambda_i(p)}. \]

\end{definition} 

We collect some basic results on the convergence of polynomials. The proofs are standard applications of Rouché's theorem.
\begin{lemma}\label{lem:conv_rel_pols}
    Let $(p_N)_N\subset \tilde{\pols}_{\le n}(\R)$, where $n$ is fixed. If $p_N \to q$ point-wise then $q\in\tilde{\pols}_{\le n}(\R)$. 

Similarly, if $(p_N)_N\subset\tilde{\pols}_{\le n}(\R_{\ge0})$ and $p_N \to q$ point-wise, then $q\in\tilde{\pols}_{\le n}(\R_{\ge0})$. 
\end{lemma}

\begin{definition}[Polynomials and their coefficients]\label{e_k}
Every polynomial $p\in \tilde\pols_{\leq n}(\C)$ can be written uniquely in the form
\begin{align*}
    p(z) &= \sum_{k=0}^n (-1)^k \binom{n}{k} \coef{k}{n}(p) z^{n-k}.
\end{align*}
In particular, if $p(z) = \prod_{i=1}^n (z-\lambda_i) \in \pols_n(\C)$, then
\[ \coef{k}{n}(p) = \binom{n}{k}^{-1} \sum_{1 \le i_1 < \dots < i_k \le n} \lambda_{i_1} \cdots \lambda_{i_k} \]
for $k\in\{1,\dots,n\}$, while $\coef{0}{n}(p) = 1$.

We say that a polynomial $p\in\tilde\pols_{\leq n}(\C)$ is $d$-monic, $d\in\{0,\ldots,n\}$, if $\coef{d}{n}(p)=1$. In particular, monic polynomials are $0$-monic. More precisely, this notion depends on both \(n\) and \(d\). We suppress the parameter \(n\) when it is clear from the context.
\end{definition}

\subsection{Finite free convolutions}

Finite free convolutions are binary operations on polynomials with remarkable root-preserving properties. The symmetric additive and multiplicative versions were studied classically by Walsh \cite{walsh1922location} and Szeg\H{o} \cite{szego1922bemerkungen}. More recently, Marcus, Spielman, and Srivastava \cite{MSS} showed that these convolutions admit a representation in terms of expected characteristic polynomials of random matrices and developed their connection with free probability.

\begin{definition}[Finite Free Convolutions]\label{def:finite_free_conv}
Let $p, q \in \tilde\pols_{\leq n}(\C)$.
\begin{itemize}
\item Their finite free additive convolution, denoted $p \boxplus_n q$, is the polynomial in $\tilde\pols_{\leq n}(\C)$ defined by its coefficients:
\[ \coef{k}{n}(p \boxplus_n q) = \sum_{j=0}^k \binom{k}{j} \coef{j}{n}(p) \coef{k-j}{n}(q) \quad \text{for } k=0, 1, \dots, n. \]
    \item Their finite free multiplicative convolution, denoted $p \boxtimes_n q$, is the polynomial in $\tilde\pols_{\leq n}(\C)$ defined by its coefficients:
\[ \coef{k}{n}(p \boxtimes_n q) = \coef{k}{n}(p) \coef{k}{n}(q) \quad \text{for } k=0, 1, \dots, n. \]
\end{itemize}
\end{definition}
\begin{corollary} 
The following properties follow directly from the definition.
    \begin{itemize}
        \item[(i)] the finite free additive and multiplicative convolutions are both bilinear, commutative and associative,
         \item[(ii)] the neutral element of $\boxplus_n$ is the polynomial $e_{\boxplus_n}(x)=x^n$,
        \item[(iii)] the neutral element of $\boxtimes_n$ is the polynomial $e_{\boxtimes_n}(x)=(x-1)^n$. 
    \end{itemize}
\end{corollary}

The following preservation property is standard in finite free probability. While closely related statements appear in the literature, we could not find a reference for the precise formulation needed below, so we include a short proof for completeness.
\begin{proposition}[Root-preserving properties]\label{prop:pres}
The finite free convolutions preserve nonnegative real-rootedness in the following sense: if
\(p,q\in \pols_n(\R_{\ge 0})\), then
\[
p\boxplus_n q,\; p\boxtimes_n q \in \pols_n(\R_{\ge 0}).
\]
\end{proposition}

\begin{proof}
The corresponding property for the finite free multiplicative convolution is stated explicitly, for example, in \cite[Section 2.5]{AGP23}. It is also noted there that if \(p,q\in\pols_n(\R)\), then
\[
p\boxplus_n q\in\pols_n(\R).
\]
Thus, if \(p,q\in\pols_n(\R_{\ge 0})\), then in particular $p\boxplus_n q\in\pols_n(\R)$.
It remains to show that all roots of \(p\boxplus_n q\) are nonnegative.

Since \(p,q\in\pols_n(\R_{\ge 0})\), Vieta's formulas imply that
\[
\coef{k}{n}(p)\ge 0,
\qquad
\coef{k}{n}(q)\ge 0,
\qquad k=0,\dots,n.
\]
By the coefficient formula for the finite free additive convolution,
\[
\coef{k}{n}(p\boxplus_n q)
\ge 0,
\qquad k=0,\dots,n.
\]
We have
\[
(p\boxplus_n q)(-z)
= (-1)^n
\sum_{k=0}^n \binom{n}{k}\coef{k}{n}(p\boxplus_n q)\, z^{n-k}.
\]
By Descartes' rule of signs, the degree-$n$ polynomial \((p\boxplus_n q)(-z)\) has no positive roots, and hence \(p\boxplus_n q\) has no negative roots. 
\end{proof}

\subsection{Normalized derivatives and reduction to equal degrees}
Although Definition \ref{def:finite_free_conv} allows for polynomials of different degrees, it is more convenient to work with those of the same degree. In what follows, we describe the reduction to the same-degree case.
\begin{definition}[Normalized derivatives]
Let $p \in \tilde{\pols}_{\le n}(\C)$. We define two normalized derivative operators on $p$:

\begin{itemize}
    \item[(i)] The normalized standard derivative, denoted by $\partial_n$, is given by
    \[
        \partial_n p(z) := \frac{p'(z)}{n}.
    \]

    \item[(ii)] The normalized polar derivative (with respect to $0$), denoted by $D_n$, is given by
    \[
        D_n p(z) := \frac{z p'(z)}{n} - p(z).
    \]
\end{itemize}

More generally, for each integer $d$ with $0 \le d \le n$, the $d$th normalized standard derivative $\partial_n^d$ and the $d$th normalized polar derivative $D_n^d$ are defined recursively by
\[
    \partial_n^d := \partial_{n-d+1} \circ \partial_n^{d-1},
    \qquad
    D_n^d := D_{n-d+1} \circ D_n^{d-1},\qquad d\geq 1,
\]
    where $\partial_n^0$ and $D_n^0$ are the identity operators. 
\end{definition}

\begin{proposition}\label{prop:conv_diff_deg}
    Let $p\in\pols_n(\C)$, $q\in\pols_{n-d}(\C)$. Then
    \[
    p\boxplus_n q =\partial_n^dp\boxplus_{n-d}q \qquad\mbox{and}\qquad p\boxtimes_nq=D_n^dp \boxtimes_{n-d}q.
    \]
\end{proposition}
\begin{proof}
    The first identity follows by applying Lemma 1.16 from \cite{MSS} successively $d$ times.

    For the second identity, we apply $d$ times the degree-reduction formula for the multiplicative convolution from \cite[Lemma 4.9]{MSS}. We note that, in the notation of \cite{MSS}, there is a sign typo in Lemma 4.9: the correct degree-reduction operator is $xD-d$. Passing to our normalization, this operator is exactly
    \[
    D_n p(z)=\frac{zp'(z)}{n}-p(z).
    \]
    Hence each application lowers the degree by one, and after $d$ steps we obtain
    \[
    p\boxtimes_n q = D_n^d p \boxtimes_{n-d} q.
    \]
    This completes the proof.
\end{proof}

\begin{lemma}\label{lemma:normalized_derivatives_coefficients}
    Let $p \in \tilde{\pols}_{\le n}(\C)$ be written in the form 
    \[
    p(z) = \sum_{k=0}^n (-1)^k \binom{n}{k} \coef{k}{n}(p) z^{n-k}.
    \]
    Then, for any $0 \le d \le n$, the $d$-fold normalized derivatives satisfy:
            \[ \partial_n^d p(z) = \sum_{k=0}^{n-d} (-1)^k \binom{n-d}{k} \coef{k}{n}(p) z^{n-d-k}. \]
            and
        \[ D_n^d p(z) = \sum_{k=0}^{n-d} (-1)^k \binom{n-d}{k} \coef{k+d}{n}(p) z^{n-d-k}. \]
\end{lemma}

\begin{proof}
We prove the formula for \(D_n^d\); the proof for \(\partial_n^d\) is analogous.
The case \(d=0\) is immediate. Assume the claim holds for \(d-1\), and set
\(m:=n-d+1\). Then
\[
D_n^{d-1}p(z)
=
\sum_{k=0}^{m}
(-1)^k \binom{m}{k}\coef{k+d-1}{n}(p)\,z^{m-k}.
\]
Applying \(D_m\), we get
\begin{align*}
D_n^d p(z)
&=D_m(D_n^{d-1}p)(z)=
\sum_{k=0}^{m}
(-1)^k \binom{m}{k}
\left(\frac{m-k}{m}-1\right)
\coef{k+d-1}{n}(p)\,z^{m-k}\\
&=
-\sum_{k=1}^{m}
(-1)^k \frac{k}{m}\binom{m}{k}\coef{k+d-1}{n}(p)\,z^{m-k}.
\end{align*}
Using $\frac{k}{m}\binom{m}{k}=\binom{m-1}{k-1}$, and reindexing with \(j=k-1\), we obtain
\begin{align*}
D_n^d p(z)
&=
\sum_{k=1}^{m}
(-1)^{k-1}\binom{m-1}{k-1}\coef{k+d-1}{n}(p)\,z^{m-k}=
\sum_{j=0}^{m-1}
(-1)^j \binom{m-1}{j}\coef{j+d}{n}(p)\,z^{m-1-j}.
\end{align*}
Since \(m-1=n-d\), we obtain the assertion. 
\end{proof}

Recall that $\pols_n(\C)$ denotes the set of all monic polynomials of degree $n$ and $p\in\tilde\pols_{\leq n}(\C)$ is $d$-monic if $\coef{d}{n}(p)=1$ (Definition \ref{e_k}). The following result is immediate. 
\begin{corollary}\label{cor:monic_deriv}Let $d\in\{1,\ldots,n\}$ and $p\in\pols_{n}(\C)$. Then $\partial_n^d p\in\pols_{n-d}(\C)$. 

Moreover, if $p$ is $d$-monic then $D_n^d p\in\pols_{n-d}(\C)$.
\end{corollary}

The following result is standard in the literature \cite{Rahman}. We also apply Corollary \ref{cor:monic_deriv} to ensure monicity of $D_n^d p\in\pols_{n-d}$.
\begin{proposition}\label{prop:derivatives_preserve_nonnegative_roots}
Let $d\in\{1,\ldots,n\}$ and $p\in\pols_{n}(\R_{\ge0})$. Then $\partial_n^d p\in\pols_{n-d}(\R_{\ge0})$. 

Moreover, if $p$ is $d$-monic then $D_n^d p\in\pols_{n-d}(\R_{\ge0})$.
\end{proposition}

\subsection{Approximation of free convolutions} 
Finite free convolutions can be used to approximate convolutions in free probability: as the degree of the polynomials tends to infinity, the empirical measures of convolutions of polynomials converge weakly to their counterparts in free probability. In this section, we give the precise statements of some known results.

\subsubsection{Free probability transforms and free convolutions}\label{sec:tr}

Free additive and multiplicative convolutions (denoted by $\boxplus$ and $\boxtimes$, respectively) are the free-probability analogues of addition and multiplication. In this subsection we recall their analytic definitions in terms of the $R$- and $S$-transforms.

For $\mu\in\MM(\R)$, the Cauchy transform is
\[
  G_\mu(z)=\int_{\R}\frac{1}{z-t}\,\mu(\dt),\qquad z\in\C\setminus\R.
\]
It is analytic on $\C\setminus\R$ and satisfies $G_\mu(z)=z^{-1}+O(|z|^{-2})$ as $|z|\to\infty$.
It is known that $G_\mu$ admits a unique compositional inverse $K_\mu$ on a suitable domain
$\Gamma_\mu$ (a neighborhood where the inverse is defined), so that
\[
  G_\mu\bigl(K_\mu(z)\bigr)=z,\qquad z\in\Gamma_\mu.
\]
The $R$-transform of $\mu$ is defined by
\[
  R_\mu(z)=K_\mu(z)-\frac1z,\qquad z\in\Gamma_\mu.
\]
Given $\mu,\nu\in\MM(\R)$, their free additive convolution $\mu\boxplus\nu$ is the unique probability
measure on $\R$ satisfying
\[
  R_{\mu\boxplus\nu}(z)=R_\mu(z)+R_\nu(z),\qquad z\in\Gamma_\mu\cap\Gamma_\nu.
\]

Let $\mu\in\MM(\R_{\ge 0})$. The moment transform of $\mu$ is
\[
  \Psi_\mu(z):=\int_{0}^{\infty}\frac{t z}{1-t z}\,\mu(\dt),\qquad z\in\C\setminus\R_{\ge 0}.
\]
If $\mu\neq\delta_0$, then the compositional inverse $\Psi_\mu^{-1}$ exists in a neighborhood of
$(-1+\mu(\{0\}),0)$. The $S$-transform of $\mu$ is defined by
\[
  S_\mu(z):=\frac{1+z}{z}\,\Psi_\mu^{-1}(z),\qquad z\in(-1+\mu(\{0\}),0).
\]

If $\mu,\nu\in\MM(\R_{\ge 0})\setminus\{\delta_0\}$, their free multiplicative convolution $\mu\boxtimes\nu$
is the unique probability measure on $\R_{\ge 0}$ such that, on a common domain where the transforms
are defined,
\[
  S_{\mu\boxtimes\nu}(z)=S_\mu(z)\,S_\nu(z).
\]
\begin{remark}\label{rem:free}
Let $(\mathcal A,\tau)$ be a tracial $W^*$-probability space, that is, a von Neumann algebra
endowed with a faithful normal tracial state $\tau$, and let $\mu_\X$ denote the spectral distribution of a self-adjoint (possibly unbounded, affiliated) operator $\X$. If $\X$ and $\Y$
are self-adjoint and free (in the sense of Voiculescu), then
\[
  \mu_{\X+\Y}=\mu_\X\boxplus\mu_\Y.
\]
If $\X,\Y\ge 0$ are free, then
\[
  \mu_{\X^{1/2} \Y \X^{1/2}}=\mu_\X\boxtimes\mu_\Y.
\]
\end{remark}

\subsubsection{Weak convergence and relation to free convolutions}

Both finite free convolutions approximate their free counterparts. These results will be crucial in our link between finite free perpetuities and free perpetuities studied in \cite{FreePerp}.

For compactly supported limiting measures, these approximation results were previously obtained in the additive case by Arizmendi and Perales \cite[Theorem~5.2 and Corollary~5.3]{AP18}, and in the multiplicative case by Arizmendi, Garza-Vargas and Perales \cite[Theorem~1.4]{AGP23}. More recently, Jalowy, Kabluchko and Marynych gave alternative proofs of the convergence $\boxplus_n\to\boxplus$ and $\boxtimes_n\to\boxtimes$ using exponential profiles of polynomial coefficients; see \cite[Theorems~3.5 and~3.7]{JKM25}. We use below the following more general formulations.

\begin{theorem}[\cite{Fujie25}, Theorem 1.3]
\label{thm:fujie}
Let $(p_n)_{n\in\N}$ and $(q_n)_{n\in\N}$ be sequences of polynomials in $\pols_n(\R)$ such that $\meas{p_n} \weak \mu$ and $\meas{q_n} \weak \nu$ with $\mu, \nu \in \MM(\R)$. Then, 
\[
\meas{p_n \boxplus_n q_n}\weak \mu \boxplus \nu.
\]
\end{theorem}

\begin{theorem}[\cite{AFPU24}, Proposition 10.1]
\label{prop:approx_boxtimesn}
Let $(p_n)_{n\in \N}$ and $(q_n)_{n\in \N}$ be sequences of polynomials in $\pols_n(\R_{\geq 0})$ such that $\meas{p_n} \weak \mu$ and $\meas{q_n} \weak \nu$ with $\mu, \nu \in \MM(\R_{\geq 0})$. Then,
\[ \meas{p_n \boxtimes_n q_n} \weak \mu \boxtimes \nu. \]
\end{theorem}

Using the above results, we can describe the weak convergence of empirical root distributions of $k$fold finite convolutions. 
\begin{corollary}
\label{cor:iterated_approximation}
Let $\kappa\in\N$ and let $\mathcal{I}\subset\{1,\dots,\kappa\}.$ For $n\in\N$, denote
\[
q_n
=
p_{n}^{[\kappa]}\boxast_n^{(\kappa)} \left( p_{n}^{[\kappa-1]}\boxast_n^{(\kappa-1)}\left(\cdots p_n^{[2]}\boxast_n^{(2)}\left( p_{n}^{[1]}\boxast_n^{(1)} p_n^{[0]}\right)\cdots \right)\right),
\]
where \(p_m^{[i]}\in \pols_m(\R_{\geq 0})\) for \(i=0,\ldots,\kappa\) and for $m\in\N$, 
\[
\boxast_m^{(i)} =
\begin{cases}
  \boxtimes_m,       & \text{if } i \in\mathcal{I}, \\
  \boxplus_m,      & \text{if } i \notin \mathcal{I}.
\end{cases}
\]
Assume that
\[
\meas{p_n^{[i]}}\weak \mu_i, \qquad i=0,\ldots,\kappa,
\]
for some \(\mu_i\in \MM(\R_{\geq 0})\). Then
\[
\meas{q_n}\weak \mu_\kappa\boxast^{(\kappa)} \left( \mu_{\kappa-1}\boxast^{(\kappa-1)}\left(\cdots \mu_2 \boxast^{(2)}\left( \mu_1\boxast^{(1)} \mu_0\right)\cdots \right)\right),
\]
where
\[
\boxast^{(i)} =
\begin{cases}
  \boxtimes,       & \text{if } i \in\mathcal{I}, \\
  \boxplus,      & \text{if } i \notin \mathcal{I}.
\end{cases}
\]
\end{corollary}

\begin{proof}
Set \(r_n^{[0]}:=p_n^{[0]}\) and recursively
\[
r_n^{[j]}:=p_n^{[j]}\boxast_n^{(j)} r_n^{[j-1]}, \qquad j=1,\dots,\kappa.
\]
Also define \(\rho^{[0]}:=\mu_0\) and
\[
\rho^{[j]}:=\mu_j\boxast^{(j)}\rho^{[j-1]}, \qquad j=1,\dots,\kappa.
\]
By Proposition \ref{prop:pres}, $r_n^{[j]}\in\pols_n(\R_{\geq 0})$ for $j=0,1,\ldots,\kappa$. Using Theorem~\ref{thm:fujie} and Theorem~\ref{prop:approx_boxtimesn} and induction on $j$, we obtain 
\[
\meas{r_n^{[j]}}\weak \rho^{[j]},\quad j=1,\ldots,\kappa. 
\]
For \(j=\kappa\) we obtain the desired convergence.
\end{proof}

There is a convenient tool for proving weak convergence of $\meas{p_n}$; following \cite{AFPU24}, we introduce the finite $S$-transform and explain its relation to weak convergence below.

\begin{definition}[Finite $S$-transform]
\label{def:finite.Strans}
Let $p \in \polplus$ such that $p(x)\neq x^n$ and $r_n$ the multiplicity of the root $0$ in $p$. The finite $S$-transform of $p$ is the map 
\[S_p^{(n)} \colon \left\{-\frac{k}{n}\colon  k =1,2,\dots, n-r_n \right\} \to \R_{> 0}\] such that
\begin{equation}
\label{eq:def.Strans}
\strans{p}{n}\left(-\frac{k}{n}\right) := \frac{\coef{k-1}{n}(p)}{\coef{k}{n}(p)} \quad \text{ for } \quad k=1,2, \dots, n-r_n.
\end{equation}
\end{definition}

\begin{theorem}
\label{thm:main} 
Let $(p_n)_{n\in \N}$ be a sequence of polynomials with $p_n\in \pols_n(\R_{\geq 0})$ and $\mu\in \MM(\R_{\geq 0})\backslash \{\delta_0\}$. The following are equivalent:
\begin{enumerate}
\item The weak convergence: $\meas{p_n} \weak \mu$. 
\item For every $t\in (0,1-\mu(\{0\}))$ and every sequence $(k_n)_{n\in\N}$ with 
\[
1\leq k_n \leq n\quad\mbox{and}\quad\lim_{n\to\infty} \frac{k_n}{n}=t,
\]
one has 
\begin{equation}
\label{eq:intro.limit.S}
\lim_{\substack{n\to \infty}}\strans{p_n}{n}\left(-\frac{k_n}{n}\right)= S_\mu(-t),
\end{equation}
where $S_\mu$ is the $S$-transform of $\mu$. 
\end{enumerate}
\end{theorem}

\section{Motivating example}\label{sec:ex}
\subsection{Jacobi polynomials}\label{sec:hyp}

Below we define a family of polynomials that will serve as a motivating example for our subsequent results. This family forms a special case of the class of Jacobi polynomials; for a general definition of hypergeometric polynomials, we refer to \cite{MMP23}; see also \cite{JKM25b}.

\begin{definition}
Let $n\in \N$, $a\in \C$ and $b\in \C\setminus \{1,1-\frac1n,1-\frac2n,\dots,\frac1n\}$.
We define $p_{n,(a,b)}\in \pols_n(\C)$ by
\[
\coef{k}{n}\!\left(p_{n,(a,b)}\right)
= (-1)^k\,\frac{(-an)_k}{((b-1)n)_k},
\qquad k=0,1,\dots,n,
\]
where $(x)_0:=1$ and
\[
(x)_k = x(x+1)\cdots(x+k-1), \qquad k\ge 1,
\]
denotes the rising factorial.
\end{definition}

    Using the terminology and notation of \cite[Section 5.2]{MMP23}, $p_{n,(a,b)}$ is a Jacobi polynomial and we have 
\[
p_{n,(a,b)}(z) = (-1)^n\HGP{n}{a}{1-b}(-z).
\]

\begin{proposition}\label{thm:hypergeom_zeros}
We have
\[
p_{n,(a,b)}(z)
= z^n\,{}_2F_1\!\left(-n,-an;(b-1)n;-z^{-1}\right).
\]
Moreover, if $a>1-\frac1n$ and $b>1$, then
\[
p_{n,(a,b)} \in \pols_n(\R_{>0}).
\]
\end{proposition}

\begin{proof}
The hypergeometric representation is immediate from the definition:
\[
p_{n,(a,b)}(z)
=\sum_{k=0}^n \binom{n}{k}\frac{(-an)_k}{((b-1)n)_k}z^{n-k}
= z^n\,{}_2F_1\!\left(-n,-an;(b-1)n;-z^{-1}\right).
\]
Denote 
\[
q(z):={}_2F_1(-n,-an;(b-1)n;z)
\]
Since
\[
p_{n,(a,b)}(z)=z^n q(-z^{-1}),
\]
a number $\lambda$ is a root of $q$ if and only if it is different from $0$ and $-\lambda^{-1}$ is a root of $p_{n,(a,b)}$.

It follows from \cite[Theorem 1(i)]{RealZeros} that if $a>1-\frac1n$ and $b>1$, the polynomial $q$ has all its roots in $\R_{<0}$.  Thus, all roots of $p_{n,(a,b)}$ are positive.
\end{proof}

Let
\[
b_i = 1-\frac{i-1}{n},\qquad i=1,\ldots,n.
\]
As \(b\to b_i\), the normalized coefficients
\[
\coef{k}{n}(p_{n,(a,b)})
=
(-1)^k\frac{(-an)_k}{((b-1)n)_k}
\]
remain finite for \(k<i\), while for \(k\ge i\) they have a simple pole provided \((-an)_k\neq 0\). 

Assume that \((-an)_i\neq0\). Then, as \(b\to b_i\), exactly \(i\) roots of
\(p_{n,(a,b)}\) escape to infinity, while the remaining \(n-i\) roots converge
to finite complex numbers, counted with multiplicity. Consequently,
\[
\meas{p_{n,(a,b)}} \vague \nu_{n,(a,b_i)},
\]
where \(\vague\) denotes vague convergence and
\(\nu_{n,(a,b_i)}\) is a subprobability measure with
\[
\nu_{n,(a,b_i)}(\C)=\frac{n-i}{n}.
\]
The finite limiting roots will be identified below as the roots of the 
degree \(n-i\) monic polynomial: we extend the definition of \(p_{n,(a,b)}\) to \(b=b_i\) by requiring that
\[
\meas{p_{n,(a,b_i)}}=\frac{n}{n-i}\nu_{n,(a,b_i)},
\qquad
p_{n,(a,b_i)}\in\pols_{n-i}(\C).
\]
For the moment assume that $(-an)_i\neq 0$,
so that $\coef{i}{n}(p_{n,(a,b)})\neq 0$. Then
\begin{align}\label{eq:defb=1}
p_{n,(a,b_i)}(z):=
\lim_{b\to b_i}
\frac{
p_{n,(a,b)}(z)
}{
(-1)^i\binom{n}{i}\coef{i}{n}(p_{n,(a,b)}).
}
\end{align}
The coefficients in the last expression are polynomial functions of \(a\), see below.
Therefore this formula extends uniquely to all \(a\in\C\), and we use it as
the definition of \(p_{n,(a,b_i)}\) for arbitrary \(a\).
\begin{proposition}\label{prop:b_i}
Fix $i\in\{1,\ldots,n\}$ and $b_i=1-\frac{i-1}{n}$ and \(a\in\C\). We have
\[
\coef{k}{n-i}(p_{n,(a,b_i)})
= (-1)^k
\frac{\falling{i-an}{k}}{(i+1)_k},\qquad k=0,\ldots,n-i.
\]
Equivalently,
\[
p_{n,(a,b_i)}(z)
=
z^{n-i}{}_2F_1(i-n,i-an;i+1;-z^{-1}).
\]
Moreover, if \(a>1-\frac1n\), then
\[
p_{n,(a,1)}\in\pols_{n-1}(\R_{\ge0}).
\]
\end{proposition}

\begin{proof}
Set $\alpha=(b-1)n$. 
Then \(b\to b_i\) is equivalent to $\alpha\to 1-i$.
For \(b\notin \{1,1-\frac1n,\ldots,\frac1n\}\), we have
\[
p_{n,(a,b)}(z)
=
\sum_{\ell=0}^n
\binom{n}{\ell}
\frac{(-an)_\ell}{(\alpha)_\ell}
z^{n-\ell}.
\]
Assume first that \((-an)_i\neq0\). 
Since
\[
(-1)^i\binom{n}{i}\coef{i}{n}\!\left(p_{n,(a,b)}\right)
=
\binom{n}{i}\frac{(-an)_i}{(\alpha)_i},
\]
by \eqref{eq:defb=1}, we obtain
\[
p_{n,(a,b_i)}(z)
=
\lim_{\alpha\to 1-i}
\sum_{\ell=0}^n
\frac{
\binom{n}{\ell}\frac{(-an)_\ell}{(\alpha)_\ell}
}{
\binom{n}{i}\frac{(-an)_i}{(\alpha)_i}
}
z^{n-\ell}.
\]
For \(\ell<i\), the numerator stays finite, while the denominator has a simple
pole. Hence those terms vanish in the limit. Therefore only the terms
\(\ell=i+k\), \(k=0,\ldots,n-i\), survive. For such \(\ell\), we have
\[
\frac{(-an)_{i+k}}{(-an)_i}=(i-an)_k\quad\mbox{and}\quad 
\frac{(\alpha)_i}{(\alpha)_{i+k}}
=
\frac{1}{(\alpha+i)_k}
\longrightarrow
\frac1{(1)_k}
=
\frac1{k!}.
\]
Consequently,
\begin{align*}
p_{n,(a,b_i)}(z)
&=
\sum_{k=0}^{n-i}
\frac{\binom{n}{i+k}}{\binom{n}{i}}
\frac{(i-an)_k}{k!}
z^{n-i-k}.
\end{align*}
Using
\[
\frac{\binom{n}{i+k}}{\binom{n}{i}}\frac1{k!}
=
\frac{\binom{n-i}{k}}{(i+1)_k},
\]
we get
\[
\coef{k}{n-i}\!\left(p_{n,(a,b_i)}\right)
=(-1)^k
\frac{\falling{i-an}{k}}{(i+1)_k},
\qquad k=0,\ldots,n-i.
\]
This coefficient formula is a polynomial expression in \(a\). Therefore the formula extends
uniquely to all \(a\in\C\) and we have $p_{n,(a,b_i)}\in\pols_{n-i}(\C)$ for all $a\in\C$. 

The hypergeometric representation of $p_{n,(a,b_i)}$ follows from the definition of ${}_2F_1$.

Finally, assume \(a>1-\frac1n\). For every \(b>1\), Proposition~\ref{thm:hypergeom_zeros} gives 
\[ 
\frac{
p_{n,(a,b)}(z)
}{
-n\coef{1}{n}(p_{n,(a,b)}).
}\in \tilde{\pols}_{\le n}(\R_{>0}). 
\] These polynomials converge point-wise to \(p_{n,(a,1)}\), so by Lemma~\ref{lem:conv_rel_pols}, \[ p_{n,(a,1)}\in \tilde{\pols}_{\le n}(\R_{\ge0}). \] But we already know that $p_{n,(a,1)}\in \pols_{n-1}(\C)$, thus $p_{n,(a,1)}\in \pols_{n-1}(\R_{\geq 0})$.
\end{proof}

\subsection{Weak convergence of Jacobi polynomials}\label{sec:WeakConvJacobi}
Yoshida in \cite{Yoshida} defined the free-beta prime distribution $f\mathcal{B}'_{a,b}$ as the free multiplicative convolution $\mu_a\boxtimes \mu_b^{-1}$, $a>0, b>1$, where $\mu_b^{-1}$ is the pushforward measure of $\mu_b$ by the mapping $x\mapsto x^{-1}$, and $\mu_\lambda$ is the Marchenko-Pastur distribution. In \cite{FreePerp}, it was observed that this definition makes sense also for $b=1$. 
For $a>0$ and $b\geq 1$, this measure is uniquely determined by its $S$-transform 
 \[
 S_{f\mathcal{B}'_{a,b}}(z)=\frac{b-1-z}{a+z},\quad z\in(-\min\{a,1\},0). 
 \]

The following result was already announced in \cite[Section 5.3]{MMP23}. 
 \begin{proposition}\label{pro:conv_pn}
For $a\geq 1$ and $b\geq 1$, 
\[
\meas{p_{n, (a, b)}} \weak f\mathcal{B}'_{a,b}.
\]
 \end{proposition}
\begin{proof}
Fix $a\geq 1$. First assume $b>1$, and set $p_n:=p_{n,(a,b)}$.
By definition,
\[
\coef{k}{n}(p_n)=(-1)^k\,\frac{(-an)_k}{((b-1)n)_k},
\qquad k=0,1,\dots,n,
\]
Since $b>1$, for all $k=0,\ldots,n$, the denominator is not zero.
Therefore $\coef{k}{n}(p_n)=0$ holds if and only if $\falling{-an}{k}=0$, which happens exactly when
$an\in\{0,1,\dots,n-1\}$, which is prohibited by $a\geq 1$. Hence the multiplicity of the root $0$ in $p_{n,(a,b)}$ equals $r_n=0$.

For every $k$ in the domain of the finite $S$-transform, i.e.\ $k=1,\dots,n$, we compute
\begin{align*}
\strans{p_n}{n}\!\left(-\frac{k}{n}\right)
&=\frac{\coef{k-1}{n}(p_n)}{\coef{k}{n}(p_n)}
= -\,\frac{\falling{-an}{k-1}}{\falling{-an}{k}}\,
   \frac{\falling{(b-1)n}{k}}{\falling{(b-1)n}{k-1}} 
= \frac{(b-1)n+k-1}{an-k+1}.
\end{align*}

Let $t\in(0,1)$.
Take any sequence $(k_n)$ with $k_n/n\to t$. Then,
\[
\lim_{n\to\infty}\strans{p_n}{n}\!\left(-\frac{k_n}{n}\right)
=\lim_{n\to\infty}\frac{(b-1)n+k_n-1}{an-k_n+1}
=\frac{b-1+t}{a-t}.
\]
On the other hand, for $z=-t$ in the domain,
\[
S_{f\mathcal{B}'_{a,b}}(-t)=\frac{b-1+t}{a-t}.
\]
Thus condition (2) of Theorem~\ref{thm:main} holds with $\mu=f\mathcal{B}'_{a,b}$, and
Theorem~\ref{thm:main} yields 
\[
\meas{p_n}\weak f\mathcal{B}'_{a,b}.
\]

If $b=1$, then $p_{n,(a,1)}\in\pols_{n-1}(\R_{\geq 0})$ and 
\[
\strans{p_n}{n-1}\!\left(-\frac{k}{n-1}\right)
=\frac{\coef{k-1}{n-1}(p_n)}{\coef{k}{n-1}(p_n)} = -\frac{(1-an)_{k-1}}{(1-an)_k}\frac{(k+1)!}{k!} = \frac{k+1}{a n-k}.
\]
Thus, if $k_n/(n-1)\to t$, then 
\[
\lim_{n\to\infty} \strans{p_n}{n-1}\!\left(-\frac{k_n}{n-1}\right)= \frac{t}{a-t} = S_{f\mathcal{B}'_{a,1}}(-t).
\]
Again, Theorem~\ref{thm:main} yields
\[
\meas{p_{n,(a,1)}}\weak f\mathcal B'_{a,1},
\]
which completes the proof. 
\end{proof}

\subsection{A fixed point equation}

Let $n\in\N$  and $a,b\in\C$ be such that
\begin{align}\label{eq:assum_minimal}
a+b\notin\left\{1,1-\frac1n,\ldots,\frac1n\right\}.
\end{align}
Consider the fixed-point equation
\begin{equation}\label{eq:fixedq}
q_n
=
p_{n,(a,a+b)}\boxtimes_n (e_{\boxtimes_n} \boxplus_n q_n),
\qquad q_n\in\tilde{\pols}_{\le n}(\C),
\end{equation}
where
\[
e_{\boxtimes_n}(z)=(z-1)^n.
\]
We now give an explicit solution to \eqref{eq:fixedq}. If \(b\notin\{1,1-\frac1n,\ldots,\frac1n\}\), then the solution has degree \(n\).
If \(b=b_i=1-\frac{i-1}{n}\) for some \(i\in\{1,\ldots,n\}\), then the solution has degree \(n-i\).

\begin{proposition}\label{prop:pff_shift_factorization}
Assume \eqref{eq:assum_minimal}.
The polynomial \(q_n=p_{n,(a,b)}\) solves
\eqref{eq:fixedq}.
\end{proposition}

\begin{proof}
We first treat the case
\(
b\notin\{1,1-\frac1n,\ldots,\frac1n\}.
\)
Set
\[
q_n:=p_{n,(a,b)}.
\]
The shift property of the finite free additive
convolution gives
\[
e_{\boxtimes_n}\boxplus_n q_n = q_n(\,\cdot-1).
\]
By Proposition~\ref{thm:hypergeom_zeros},
\[
q_n(z-1)
=
(z-1)^n\,{}_2F_1\!\left(-n,-an;(b-1)n;-\frac{1}{z-1}\right).
\]
Applying Pfaff's transformation 
\[
{}_2F_1(A,B;C;u)
=
(1-u)^{-A}\,
{}_2F_1\!\left(A,C-B;C;\frac{u}{u-1}\right)
\]
with
\[
A=-n,\qquad B=-an,\qquad C=(b-1)n,\qquad u=-\frac{1}{z-1},
\]
which is valid here without any restriction on \(u\) since the series terminates, we obtain
\[
1-u=\frac{z}{z-1},
\qquad
\frac{u}{u-1}=\frac{1}{z},
\]
and therefore
\[
q_n(z-1)
=
z^n\,{}_2F_1\!\left(-n,(a+b-1)n;(b-1)n;z^{-1}\right).
\]
Expanding the terminating hypergeometric series yields
\[
q_n(z-1)
=
\sum_{k=0}^n
(-1)^k\binom{n}{k}
\frac{((a+b-1)n)_k}{((b-1)n)_k}
\,z^{n-k},
\]
hence
\[
\coef{k}{n}\!\bigl(q_n(\,\cdot-1)\bigr)
=
\frac{((a+b-1)n)_k}{((b-1)n)_k},
\qquad k=0,\ldots,n.
\]
On the other hand,
\[
\coef{k}{n}\!\left(p_{n,(a,a+b)}\right)
=
(-1)^k\,\frac{(-an)_k}{((a+b-1)n)_k}.
\]
Therefore, by the definition of \(\boxtimes_n\),
\begin{align*}
\coef{k}{n}\!\left(p_{n,(a,a+b)}\boxtimes_n q_n(\,\cdot-1)\right)
&=
\coef{k}{n}\!\left(p_{n,(a,a+b)}\right)\,
\coef{k}{n}\!\bigl(q_n(\,\cdot-1)\bigr)\\
&=
\left[(-1)^k\,\frac{(-an)_k}{((a+b-1)n)_k}\right]
\left[\frac{((a+b-1)n)_k}{((b-1)n)_k}\right]\\
&=
(-1)^k\,\frac{(-an)_k}{((b-1)n)_k}
=
\coef{k}{n}\!\left(p_{n,(a,b)}\right).
\end{align*}
Thus
\begin{align}\label{eq:psoln}
p_{n,(a,b)}
=
p_{n,(a,a+b)}\boxtimes_n p_{n,(a,b)}(\,\cdot-1) = p_{n,(a,a+b)}\boxtimes_n (e_{\boxtimes_n}\boxplus_n p_{n,(a,b)}),
\end{align}
which is exactly \eqref{eq:fixedq}.

It remains to consider the case \(b=b_i\) for some \(i\in\{1,\ldots,n\}\).
By \eqref{eq:assum_minimal}, we have \((-an)_i\neq0\). Hence the following
normalization is well-defined for \(b\) close to \(b_i\):
\[
r_b(z):=
\frac{
p_{n,(a,b)}(z)
}{
(-1)^i\binom{n}{i}\coef{i}{n}(p_{n,(a,b)})
}.
\]
For \(b\notin\{1,1-\frac1n,\ldots,\frac1n\}\), by \eqref{eq:psoln}, we have
\[
r_b
=
p_{n,(a,a+b)}\boxtimes_n (e_{\boxtimes_n}\boxplus_n r_b).
\]
Letting \(b\to b_i\), we have, by the definition,
$r_b\to p_{n,(a,b_i)}$.
Moreover, by \eqref{eq:assum_minimal},
$p_{n,(a,a+b)}\to p_{n,(a,a+b_i)}$.
By continuity of \(\boxtimes_n\), we obtain
\[
p_{n,(a,b_i)}
=
p_{n,(a,a+b_i)}
\boxtimes_n (e_{\boxtimes_n}\boxplus_n
p_{n,(a,b_i)}).
\]
Thus \(q_n=p_{n,(a,b_i)}\) solves \eqref{eq:fixedq}.
\end{proof}

\subsection{Weak convergence of a solution}
We show that the limit of $p_{n,(a,b)}$ as $n\to\infty$ connects to a class of free perpetuities studied in \cite{FreePerp}. Our general results in Section \ref{sec:approx} demonstrate that this relationship is, in fact, far more general.

\begin{proposition}
Fix \(a\geq 1\) and \(b\ge 1\). Then by Proposition \ref{prop:pff_shift_factorization}, $q_n=p_{n,(a,b)}$ satisfies \eqref{eq:fixedq} and by Proposition \ref{pro:conv_pn},
\begin{itemize}
    \item[(i)] $\meas{p_{n,(a,a+b)}}\weak  f\mathcal{B}'_{a,a+b}=:\mu_{\A}$,
    \item[(ii)] $\meas{q_n}\weak f\mathcal{B}'_{a,b}=:\mu_{\X}$.
\end{itemize}
Moreover, by Corollary \ref{cor:iterated_approximation} applied to  \eqref{eq:fixedq}, we obtain 
\begin{align}\label{eq:free_perp_AXA}
\mu_\X = \mu_{\A}\boxtimes (\delta_1\boxplus \mu_{\X}).
\end{align}
\end{proposition}

The fixed-point equations of the form \eqref{eq:free_perp_AXA} were studied in \cite{FreePerp}. By Remark~\ref{rem:free}, the right-hand side of \eqref{eq:free_perp_AXA} is the distribution of $\A^{1/2}(\X+1)\A^{1/2}$,
where \(\A\sim f\mathcal{B}'_{a,a+b}\) and \(\X\sim f\mathcal{B}'_{a,b}\) are free. Hence \eqref{eq:free_perp_AXA} can be written as
\[
\X \stackrel{d}{=} \A^{1/2}\X \A^{1/2} + \A,
\qquad \X \text{ and } \A \text{ are free}.
\]
This is exactly the free perpetuity considered in \cite[Section 1.5]{FreePerp}.

The behavior of the solution changes qualitatively at \(b=1\). 
If \(b>1\), then \(\mu_\X\) has compact support. At the boundary value \(b=1\), however, the behavior changes qualitatively: \(\mu_\X\) becomes heavy-tailed and
\[
\lim_{t\to\infty} t^{1/2}\,\mu_{\X}\bigl((t,\infty)\bigr)=\frac{2\sqrt a}{\pi}.
\]
Thus the tail decays as \(t^{-1/2}\). This exponent is not specific to the present example: it is the universal critical exponent for free perpetuities in the finite-variance regime. More precisely, \cite[Theorem 1.5]{FreePerp} shows that whenever
\[
\X \stackrel{d}{=} \A^{1/2}\X\A^{1/2}+\B,
\qquad m_1(\mu_\A)=1,
\]
with \(\A\) non-Dirac and of finite second moment, the law of \(\X\) has a \(t^{-1/2}\)-tail. In our example this critical regime is realized precisely at \(b=1\).

In the finite model we already see the same threshold, since
\[
m_1\left(\meas{p_{n,(a,a+b)}}\right)
=
\coef{1}{n}\!\left(p_{n,(a,a+b)}\right)
=
\frac{a}{a+b-1},
\]
so \(m_1(\mu_\A)=1\) exactly when \(b=1\).

The finite model also reflects the critical transition on the level of roots. For \(b>1\), the solution \(q_n=p_{n,(a,b)}\) has degree \(n\). At \(b=1\), after the renormalization
\[
\frac{
p_{n,(a,b)}(z)
}{
-n\coef{1}{n}(p_{n,(a,b)}).
}\to p_{n,(a,1)},
\]
the degree drops to \(n-1\). Equivalently, as \(b\downarrow1\), one root of \(q_n\) escapes to \(+\infty\), while the remaining \(n-1\) roots converge to the roots of \(p_{n,(a,1)}\). In terms of empirical root distributions, this means that \(\meas{q_n}\) loses tightness at the critical point through one escaping atom of mass \(1/n\). Thus, although each \(\meas{q_n}\) is discrete and compactly supported and therefore has no genuine tail behavior, the finite fixed-point problem already detects the same critical threshold \(m_1(\mu_\A)=1\) that in the limit produces the universal \(t^{-1/2}\) tail.

\section{The U-transform}\label{sec:U}

In this section, we recall the definition and basic properties of the $U$-transform, introduced in \cite{Mar21}, and then extend it to our setting. The 
$U$-transform provides an appealing probabilistic representation of finite free convolutions, see Lemma \ref{lem:U_transf_finite_free_conv} below. As a consequence, fixed-point equations for polynomials such as \eqref{eq:fixedq} admit a probabilistic reformulation.

Let $(\Omega, \mathcal{F},\P)$ be a probability space sufficiently rich to support any finite discrete distribution.  All random variables will implicitly be assumed to be defined on $(\Omega, \mathcal{F})$. 

\begin{definition}
Let $p\in\pols_n(\C)$ with $n\ge1$. Any complex-valued random variable $U$ with $\E[|U|^n]<\infty$ satisfying $p(z) = \E[(z-U)^n]$ will be called a $U$-transform of $p$. In this case we write $p= p_U$. 
\end{definition}

\begin{corollary}\label{cor:E[U^m]}
A $U$-transform of $p\in\pols_n(\C)$ is any complex random variable $U$ satisfying $\E[|U|^n]<\infty$ such that   $\E[U^m]=\coef{m}{n}(p_U)$ for all $m\in\{1,\dots,n\}$. In particular, $\E[U]=m_1(\meas{p})$.
\end{corollary}

\begin{remark}\label{rem:MK}
In general, the law of a $U$-transform is not unique. However, there exists a unique multiset $T$ of cardinality $n$ such that, if $U$ is uniformly distributed on $T$, then
\[
p(z)=\E[(z-U)^n];
\]
see \cite[Lemma 3.1]{Mar21}. In particular, the $U$-transform exists for all $n\ge 1$ and all polynomials $p\in\pols_n(\C)$.

If $p\in\pols_n(\R)$ and $S$ is its multiset of roots, then for a random variable $U$ uniformly distributed on a multiset $T$ of cardinality $n$, the identity
$p(z)=\prod_{s\in S}(z-s)=\E[(z-U)^n]$ 
is equivalent, for sufficiently large $x$, to
\[
\int_{\R} \log\big((x-s)^n\big)\,\mu_S({\rm d}s)
=
\log \int_{\C} (x-t)^n\,\mu_T(\dt),
\]
where for a multiset $V=\mset{\lambda_1,\dots,\lambda_n}$ we denote its empirical measure by
\[
\mu_V:=\frac1n\sum_{i=1}^n\delta_{\lambda_i}.
\]
\end{remark}

An essential property of the $U$-transform is that the additive and multiplicative convolutions of polynomials can equivalently be defined in terms of their $U$-transforms. 

\begin{proposition}
    Let $p\in\pols_n(\C)$ and $U$ be any $U$-transform of $p$. Let $d\in\{0,\dots,n-1\}$. Then 
\begin{itemize}
    \item[(i)] a $U$-transform of $\partial_n^dp$ is any complex-valued random variable $U'$ satisfying $\E[|U'|^{n-d}]<\infty$ and $\E[(U')^m]=\E[U^m]$ for all $m\in\{1,\dots,n-d\}$.
    \item[(ii)] Assume additionally that $p$ is $d$-monic. A $U$-transform of $D_n^dp$ is any complex-valued random variable $U'$ satisfying $\E[|U'|^{n-d}]<\infty$ and $\E[(U')^m]=\E[U^{m+d}]$ for all $m\in\{1,\dots,n-d\}$.
\end{itemize}
\end{proposition}
\begin{proof}
    Follows from Lemma \ref{lemma:normalized_derivatives_coefficients} and Corollary \ref{cor:E[U^m]}.
\end{proof}

\begin{lemma}\label{lem:U_transf_finite_free_conv}
    Let $p\in\pols_{n}(\C)$, $q\in\pols_{n-d}(\C)$, $d\in\{0,\ldots,n-1\}$. Let $S$ be any $U$-transform of $q$. 
\begin{itemize}
    \item[(i)] Let $T$ be any $U$-transform of $\partial_n^d p$, independent of $S$. Then
    \[
    ( p \boxplus_{n} q )(z) 
= \E[(z - (S+T))^{n-d}].
    \]
    \item[(ii)] Assume $p$ is $d$-monic and let $T$ be any $U$-transform of $D_n^d p$, independent of $S$. Then
    \[
    ( p \boxtimes_{n} q )(z) 
= \E[(z - S T)^{n-d}].
    \]
\end{itemize}
\end{lemma}
\begin{proof}
    Follows from \cite[Lemma 3.4]{Mar21} and Proposition \ref{prop:conv_diff_deg}.
\end{proof}

\begin{remark}
The conclusion of Lemma \ref{lem:U_transf_finite_free_conv} does not require full independence of the \(U\)-transforms \(S\) and \(T\).

In part \((i)\), it is sufficient to assume that their mixed moments factorize up to total order \(n-d\), namely,
\begin{align}\label{eq:add}
\E[S^iT^j]=\E[S^i]\E[T^j], \qquad i,j\ge 0,\quad i+j\le n-d.
\end{align}

In part \((ii)\), it is sufficient to assume that
\begin{align}\label{eq:mult}
\E[S^iT^i]=\E[S^i]\E[T^i], \qquad 0\le i\le n-d.
\end{align}

Indeed, part \((i)\) only involves moments of \(S+T\) up to order \(n-d\), while part \((ii)\) only involves moments of \(ST\) up to order \(n-d\).
\end{remark}

\subsection{\texorpdfstring{$n$-probability space}{n-probability space}}
Motivated by the \(U\)-transform representation of finite free convolutions, we now introduce a finite-order probabilistic formalism adapted to the parameter \(n\). Since the objects studied in this work are determined through identities involving moments only up to order \(n\), it is natural to replace the usual notions of distribution and independence by their truncated analogues. We therefore define \(n\)-distributions, \(n\)-independence, and convergence in \(n\)-distribution, which will serve as the probabilistic framework for the sequel.

All random variables considered in this section are complex-valued and have finite moments up to order \(n\).

For a random vector \(\underline{X}=(X_1,\dots,X_k)\in\C^k\), we define its \(n\)-distribution by
\[
\mathrm{Mom}_n(\underline{X})
:=
\left(
\E\!\left[\prod_{j=1}^k X_j^{i_j}\right]
\colon i_1,\dots,i_k\ge 0,\ \sum_{j=1}^k i_j\le n
\right).
\]
For random vectors \(\underline{X},\underline{Y}\in\C^k\), we write
\[
\underline{X}\overset{d}{=}_n \underline{Y},
\]
if they have the same \(n\)-distribution, that is,
\(\mathrm{Mom}_n(\underline{X})=\mathrm{Mom}_n(\underline{Y})\).

We say that random variables \(X\) and \(Y\) are \(n\)-independent if their joint moments factorize up to total order \(n\), i.e.
\[
\E[X^iY^j]=\E[X^i]\E[Y^j],
\qquad i,j\ge 0,\quad i+j\le n.
\]

We say that a sequence of random variables \((X_N)_{N\ge1}\) converges in \(n\)-distribution to \(X\) if
\[
\lim_{N\to\infty}\E[X_N^k]=\E[X^k],
\qquad k=1,\dots,n.
\]
We denote this by
\[
X_N \xrightarrow{d}_n X.
\]
We also say that a random variable \(X\) and a pair \((A,B)\) are affinely \(n\)-independent if
\[
\E[X^iA^iB^j]=\E[X^i]\E[A^iB^j],
\qquad i,j\ge 0,\quad i+j\le n.
\]
Under this condition, for every \(k=0,\dots,n\),
\[
\E[(AX+B)^k]
=
\sum_{i=0}^k \binom{k}{i}\E[X^i]\E[A^iB^{k-i}].
\]
Moreover, \eqref{eq:add} is precisely \((n-d)\)-independence of \(S\) and \(T\), and equivalently affine \((n-d)\)-independence of \(S\) and \((1,T)\). Likewise, \eqref{eq:mult} is equivalent to affine \((n-d)\)-independence of \(S\) and \((T,0)\).

\section{Finite free perpetuities}\label{sec:ffp}

Below we give a general definition of a finite free perpetuity. All random variables are complex-valued.

\begin{definition}\label{def:finite_free_perp}
  Let \(A\) and \(B\) be random variables such that
\(\E[|A|^{n}]<\infty\) and \(\E[|B|^{n}]<\infty\).
Consider the affine fixed-point equation
\begin{align}\label{eq:finite_free_perp}
X\stackrel{d}{=}_n AX+B,
\qquad
X \text{ and } (A,B) \text{ are affinely \(n\)-independent}.
\end{align}
If a random variable \(X\) satisfies \eqref{eq:finite_free_perp}, then we call the polynomial
\[
p_X(z):=\E[(z-X)^n]
\]
a finite free perpetuity polynomial of order \(n\) generated by \((A,B)\).
\end{definition}

\begin{remark}
Since
\[
p_X(z)=\E[(z-X)^n]
=\sum_{i=0}^n \binom{n}{i}(-1)^i \E[X^i]\,z^{n-i},
\]
the polynomial \(p_X\) is monic of degree \(n\), and it uniquely determines the moments
\(\E[X],\ldots,\E[X^n]\).
Now suppose \(X\) satisfies \eqref{eq:finite_free_perp}. Then
\[
p_X(z)=\E[(z-(AX+B))^n].
\]
Expanding and using affine \(n\)-independence,
\begin{align*}
p_X(z)
&=\sum_{i=0}^n \binom{n}{i}(-1)^i \E[X^i]\E\!\left[A^i(z-B)^{n-i}\right].
\end{align*}
On the other hand, on \(\{A\neq0\}\),
\[
A^n p_X\!\left(\frac{z-B}{A}\right)
=\sum_{i=0}^n \binom{n}{i}(-1)^i \E[X^i]A^i(z-B)^{n-i}.
\]
Taking expectations and adding the contribution of \(\{A=0\}\), where only the term \(i=0\) remains, yields
\[
p_X(z)
=
\E\left[A^n p_X\!\left(\frac{z-B}{A}\right)\mathbf 1_{\{A\neq0\}}\right]
+\E\left[(z-B)^n\mathbf 1_{\{A=0\}}\right].
\]
Thus every finite free perpetuity polynomial satisfies the fixed-point equation above. 
\end{remark}

\begin{example}
    Recall that in Section \ref{sec:ex} we considered a polynomial fixed point equation of the form 
\begin{align}\label{eq:fixedq2}
q_n
=
p_{n,(a,a+b)}\boxtimes_n (e_{\boxtimes_n} \boxplus_n q_n).
\end{align}
Assume $b,a+b\notin\{1,1-\frac1n,\ldots,\frac1n\}$. 
Let $A$ be any $U$-transform of $p_{n,(a,a+b)}$ and $X$ be any $U$-transform of $q_n$, so that,
 \[
    p_{n,(a,a+b)}(z) = \E[(z-A)^n]\qquad\mbox{and}\qquad     q_n(z) = \E[(z-X)^n]. 
\]
If $X$ and $(A,A)$ are affinely $n$-independent, then, by Lemma \ref{lem:U_transf_finite_free_conv}, for all $z\in\C$,
\[
\E[(z-X)^n] = q_n(z)  
=
p_{n,(a,a+b)}\boxtimes_n (e_{\boxtimes_n} \boxplus_n q_n)(z) = \E[(z-A(X+1))^n].
\]
Thus, this polynomial fixed-point equation can be reformulated probabilistically as \eqref{eq:finite_free_perp} with $B=A$. 

If $b=1$, then $q_n\in\pols_{m}(\C)$ with $m=n-1$. Then, letting $q_n(z) = \E[(z-X)^{m}]$, again by Lemma \ref{lem:U_transf_finite_free_conv},  \eqref{eq:fixedq2} is equivalent to \eqref{eq:finite_free_perp} with $n=m$ and $(A,B)=(A',A')$, where $A'$ is any $U$-transform of $D_n p_{n,(a,a+1)}$ such that $X$ and $(A',A')$ are affinely $m$-independent. We note that $A'$ is well defined, that is that $D_n p_{n,(a,a+1)}\in\pols_{m}(\C)$; it follows from Corollary~\ref{cor:monic_deriv} and the fact that $p_{n,(a,a+1)}$ is $1$-monic. 
\end{example}

 Now we establish existence and uniqueness of a finite free perpetuity. 
\begin{theorem}
\label{thm:finite_free_perp} 
Assume that $\E[A^k]\neq 1$ for all $k=1,\dots n$. Then the finite free perpetuity of order $n$  exists and is unique. 
\end{theorem}
    \begin{proof}
    From \eqref{eq:finite_free_perp} we obtain that for all $k=1,\ldots,n$,
    \begin{align*}
        \E[X^k]=\E[(AX+B)^k]
         = \sum_{i=0}^k\binom{k}{i}\E[X^i]\E[A^iB^{k-i}]
    \end{align*}
    or, equivalently, that for all $k=1,\ldots,n$,
\begin{align}\label{eq:finte_free_perp_moments}
        \E[X^k]=\frac{\sum_{i=0}^{k-1}\binom{k}{i}\E[X^i]\E[A^iB^{k-i}]}{1-\E[A^k]}
    \end{align}
Thus, the first $n$ moments of $X$ are uniquely determined by the mixed moments of $(A,B)$. Therefore, 
$p_X(z)=\E[(z-X)^n]$ exists and is unique. 
\end{proof}

\subsection{Reduction of the order of a finite free perpetuity}\label{sec:reduct}

In this section, we describe the general framework, which leads to finite free perpetuities of lower degree. Recall that in the motivating example, in the critical case \(b=1\) we had $\coef{1}{n}(p_A) = \E[A]=1$, where $p_A(z):=p_{n,(a,a+1)}(z)=\E[(z-A)^n]$. 
In that situation, the finite free perpetuity of order \(n\) does not exist, but after a suitable reduction one still obtains a meaningful object of degree \(n-1\). We now describe the general mechanism behind this phenomenon.

Set
\[
d:=\max\{k\in\{0,1,\dots,n\}\colon \E[A^k]= 1\},
\]
and assume \(0<d<n\). Then the assumptions of Theorem \ref{thm:finite_free_perp} are not satisfied. We then define \(p_{X'}\) to be the finite free perpetuity polynomial of order \((n-d)\) generated by a pair \((A',B')\) whose joint \((n-d)\)-distribution is given by
\[
\E[(A')^i(B')^j]=\E[A^{i+d}B^j],
\qquad i,j\ge 0,\quad i+j\le n-d.
\]
By construction,
\[
\E[(A')^k]=\E[A^{d+k}] \neq 1,
\qquad k=1,\dots,n-d,
\]
so Theorem~\ref{thm:finite_free_perp} applies to \((A',B')\). Hence
\[
p_{X'}(z)=\E[(z-X')^{\,n-d}]
\]
is well defined and uniquely determined.
Moreover,
\[
\E[(z-X')^{n-d}]
=\E[(z-(A'X'+B'))^{n-d}]
=\E\left[A^d(z-(AX'+B))^{n-d}\right].
\]
Thus the reduced perpetuity polynomial is obtained from the original affine recursion by weighting with the factor \(A^d\).

The construction used in the motivating example for the critical case \(b=1\) is exactly the special case \(d=1\) of the general reduction described above.
Indeed, in that example the affine recursion has the form \eqref{eq:finite_free_perp} with $B=A$. According to the general reduction, one should then pass to order \(n-1\) and replace \((A,B)\) by a pair \((A',B')\) satisfying
\[
\E[(A')^i(B')^j]=\E[A^{i+1}B^j],
\qquad i+j\le n-1.
\]
Since here \(B=A\), this becomes
\[
\E[(A')^i(B')^j]=\E[A^{i+j+1}],
\qquad i+j\le n-1.
\]
Thus it is natural to take \(B'=A'\), where \(A'\) is any random variable with
\[
\E[(A')^k]=\E[A^{k+1}],
\qquad k=0,\dots,n-1.
\]
Now let $p_A(z):=\E[(z-A)^n]$. By the formula for the normalized polar derivative,
\[
D_n p_A(z)=\E[A(z-A)^{n-1}],
\]
so the coefficients of \(D_n p_A\) are exactly the shifted moments
\[
\coef{k}{n-1}(D_n p_A)=\E[A^{k+1}], \qquad k=0,\dots,n-1.
\]
Therefore any \(U\)-transform of \(D_n p_A\) has precisely the required \((n-1)\)-distribution, and may be taken as \(A'\).

In other words, the passage from \(p_A\) to \(D_n p_A\) comes from the general rule
\[
\E[(z-X)^{n-1}]
=
\E\left[A(z-(AX+A))^{n-1}\right].
\]

\subsection{Relation to classical perpetuities}
Consider a pair of (complex-valued) random variables $(A,B)$ and the affine fixed-point stochastic equation
\begin{equation}\label{eq:classical-perpetuity}
X \stackrel{d}{=} AX+B,\qquad X\mbox{ and }(A,B)\mbox{ are independent.}
\end{equation}
A (complex-valued) random variable $X$ whose distribution satisfies
\eqref{eq:classical-perpetuity} is called a  perpetuity.

In the literature, \cite{BurBook16}, the perpetuities are usually studied for real $A,B,X$. The complex-valued case follows by identifying \(\C\) with \(\R^2\); see Appendix~\ref{app}.
In particular, if
\[
\E[\log |A|]<0
\qquad\mbox{and}\qquad
\E[\max\{\log |B|,0\}]<\infty,
\]
then the series
\[
\sum_{n=1}^\infty A_1\ldots A_{n-1} B_n,
\]
where $(A_k,B_k)_{k}$ are iid copies of $(A,B)$, 
converges almost surely and its law is the unique solution in law to \eqref{eq:classical-perpetuity}.

Note that the above conditions hold true under the stronger condition (via the Jensen inequality)
\begin{align}\label{eq:finite_moments}
\E[|A|^n]<1
\qquad\text{and}\qquad
\E[|B|^n]<\infty.
\end{align}
Moreover, \eqref{eq:finite_moments} imply that \(\E[|X|^n]<\infty\), \cite[Section 2.3]{BurBook16}.  

Comparing \eqref{eq:classical-perpetuity} with \eqref{eq:finite_free_perp}, we see that under \eqref{eq:finite_moments}, every perpetuity induces a finite free perpetuity, but the condition implying existence and uniqueness of the finite free perpetuity is weaker than the corresponding conditions for classical perpetuities. From this perspective, finite free perpetuities of order $n$ may be viewed as a reduced model in which a distribution is encoded only by its first $n$ moments, and independence is replaced by factorization of mixed moments up to order $n$. As the name suggests, finite free perpetuities are expected to approximate free perpetuities. It is therefore natural to view them as a bridge between the classical and free-probabilistic frameworks.

\subsection{Further properties of finite free perpetuities}
\begin{proposition}\label{prop:lim_E[S_N^k]}
   Let \(A\) and \(B\) be random variables such that
\(\E[|A|^{n}]<\infty\) and \(\E[|B|^{n}]<\infty\). Let 
$X$ be a solution to \eqref{eq:finite_free_perp}. 

Define   
\[
S_0=0,\qquad S_N= A_N S_{N-1}+B_N, \quad N\geq 1,
    \]
 where $(A_j,B_j)_j$ are iid copies of $(A,B)$. 
If $|\E[A^m]|<1$ for all $m=1,\ldots,n$, then as $N\to\infty$,
\[
S_N\xrightarrow{d}_n X.
\]
\end{proposition}

Before proceeding to the main proof, we establish the following elementary lemma.

\begin{lemma}\label{elementary_lemma}
    Let $(a_n)_{n \ge 1}$ be a sequence of complex numbers and let $c \in\C$ be such that $|c|<1$. Assume that the sequence $(a_n - ca_{n-1})_{n\ge 2}$ converges. Then, $(a_n)_{n\ge 1}$ converges.
\end{lemma}
\begin{proof}
The assertion holds trivially if $c=0$. We assume further that $c\neq 0$.
    Let $a_0=0$ and let $b_n=a_n-ca_{n-1}$ for $n\ge 1$. Thus, for all $n\ge 1$, 
    \[a_n=\sum_{k=1}^{n}c^{n-k}b_k.
    \]
    We denote the limit of $(b_n)_{n\ge 1}$ by $L\in\C$. Define the sequence $\varepsilon_k = b_k-L$ for all $k\ge 1$. Hence,
    \begin{align}\label{eq:an}
    a_n = L\sum_{k=0}^{n-1}c^k+\sum_{k=1}^n
c^{n-k}\varepsilon_k.    
\end{align}
Clearly, under $|c|<1$, the first term above converges as $n\to\infty$. 

Since $|c| < 1$, the sequence $|c|^{-n}$ is strictly increasing to $+\infty$. 
By the Stolz–Ces\`aro theorem,
\[
\limsup_{n \to \infty} \left| \sum_{k=1}^n c^{n-k} \varepsilon_k \right| \leq 
\lim_{n \to \infty} \frac{\sum_{k=1}^n |c|^{-k} |\varepsilon_k|}{|c|^{-n}} = 
\lim_{n \to \infty} \frac{|c|^{-(n+1)} |\varepsilon_{n+1}|}{|c|^{-(n+1)} - |c|^{-n}} = 0.
\]
Thus, the second term in \eqref{eq:an} converges to $0$ as $n\to\infty$ and the proof is complete.
\end{proof}

\begin{proof}[Proof of Proposition \ref{prop:lim_E[S_N^k]}]

Define $T_N(m)=\E[S_N^m]$, in particular $T_N(0)=1$, and $x_{i, j}=\E[B^iA^j]$. Then, by binomial expansion of $S_N^m=(A_N S_{N-1}+B_N)^m$, we obtain 
\[
T_N(m)=\sum_{i=0}^m \binom{m}{i}\,x_{i,m-i}\,T_{N-1}(m-i).
\]
Hence, for all $N\in\N$ and all $m\in\N$,
\begin{align}\label{eq:rec_T_N(k)}
T_N(m)-x_{0,m}T_{N-1}(m)=\sum_{i=1}^{m}\binom{m}{i}x_{i, m-i}T_{N-1}(m-i).
\end{align}
Using the above and Lemma \ref{elementary_lemma} we prove by induction that for all $m\in\{1,\dots,n\}$ the sequence $(T_N(m))_N$ converges. Here, we have used the assumption that $|x_{0,m}|=|\E[A^m]|<1$ in order to apply Lemma \ref{elementary_lemma}. Denote the limit by $L_m = \lim_{N \to \infty} T_N(m)$. As $N \to \infty$, \eqref{eq:rec_T_N(k)} implies that for all $m \in \{1, \dots, n\}$,
\begin{align*}
L_m=\frac{\sum_{i=1}^{m}\binom{m}{i}x_{i, m-i}L_{m-i}}{1-x_{0,m}}=\frac{\sum_{i=1}^{m}\binom{m}{i}\E[B^iA^{m-i}]L_{m-i}}{1-\E[A^m]}=\frac{\sum_{l=0}^{m-1}\binom{m}{l}\E[B^{m-l}A^l]L_{l}}{1-\E[A^m]}.
\end{align*}
From the above and \eqref{eq:finte_free_perp_moments}, an induction on $m$ yields $L_m=\E[X^m]$ for each $m \in \{1, \dots, n\}$.
\end{proof}
 
We now introduce a class of pairs \((A,B)\) that generate finite free perpetuities with real, nonnegative zeros. 

Fix $\kappa_n \in \N$ and a subset 
$\mathcal{I}_n \subseteq \{1,\dots,\kappa_n\}$. For
$d_0,\dots,d_{\kappa_n}\in\C$, we define the operator $G_{\mathcal{I}_n,\kappa_n}$ as follows:
\[
G_{\mathcal{I}_n,\kappa_n}(d_0,\dots,d_{\kappa_n})
    := (g_{\kappa_n}\circ g_{\kappa_n-1}\circ \cdots \circ g_1)(d_0),
\]
where, for each $i=1,\dots,\kappa_n$, the map 
$g_i\colon \C\to\C$ is given by
\[
g_i(z)=
\begin{cases}
d_i z, & i\in \mathcal{I}_n,\\
z+d_i, & i\notin \mathcal{I}_n.
\end{cases}
\]

\begin{definition}[$n$-admissibility]\label{def:n_admissibility} 
We say that a pair $(A_n,B_n)$ is $n$-admissible if there exist $\kappa_n\in\N$, $\mathcal{I}_n\subset \{1,\dots,\kappa_n\}$ and
independent random variables $D_1,\dots,D_{\kappa_n}$ satisfying
$\E[|D_i|^n]<\infty$ for each $i$, such that
\[
(A_n,B_n)\stackrel{d}{=}_n
\left(
\prod_{i\in\mathcal{I}_n}D_i,\,
G_{\mathcal{I}_n,\kappa_n}(0,D_1,\dots,D_{\kappa_n})
\right),
\]
that is, $A_n$ is the slope and $B_n$ is the intercept of the
random affine map $G_{\mathcal{I}_n,\kappa_n}(\cdot, D_1,\ldots,D_{\kappa_n})$.
\end{definition}

\begin{notation}
    For an $n$-admissible pair $(A_n, B_n)$ we denote \(p_n^{[i]}(z):=\E[(z-D_i)^n]\) for all $i\in\{1,\dots,\kappa_n\}$.
\end{notation}
\begin{corollary}\label{cor:n-admiss}
    Assume $(A_n, B_n)$ is $n$-admissible. Then 
    \begin{itemize}
\item[(i)] the following hold:
    \begin{align*}
        \E[A_n^m] &= \prod_{i\in\mathcal{I}_n} \coef{m}{n}(p_{n}^{[i]}), \quad \text{for all } m\in\{1,\dots,n\}, \\[1ex]
        \E[B_n] &= G_{\mathcal{I}_n,\kappa_n}(0, \coef{1}{n}(p_n^{[1]}),\dots,\coef{1}{n}(p_n^{[\kappa_n]})).
    \end{align*}
   In particular, these moments are real, if $p_n^{[i]}$ have real coefficients for all $i\in\{1,\dots,\kappa_n\}$.
\item[(ii)] for all random variable $Y_n$, which is affinely $n$-independent of $(A_n,B_n)$ one has, with $p_{Y_n}(z):= \E[(z-Y_n)^{n}]$,
\begin{align}\label{eq:AYB}
&\E[(z - (A_nY_n+B_n))^{n}] \nonumber \\
&\qquad = p_{n}^{[\kappa_n]}\boxast_n^{(\kappa_n)} \left( p_{n}^{[\kappa_n-1]}\boxast_n^{(\kappa_n-1)}\left(\cdots p_n^{[2]}\boxast_n^{(2)}\left( p_{n}^{[1]}\boxast_n^{(1)} p_{Y_n}\right)\cdots \right)\right) (z),
\end{align}
where $\boxast_n^{(i)}=\boxtimes_n$  if  \(i\in\mathcal{I}_n\) and $\boxast_n^{(i)}=\boxplus_n$ otherwise.
\item[(iii)] If $|\E[A_n^m]| <1$ for all $m=1,\ldots,n$, then the finite free perpetuity $p_{X_n}$ of order $n$ generated by $(A_n,B_n)$ is the unique polynomial of degree $n$ satisfying
    \begin{align}\label{eq:poyperp}
p_{X_n} =  p_{n}^{[\kappa_n]}\boxast_n^{(\kappa_n)} \left( p_{n}^{[\kappa_n-1]}\boxast_n^{(\kappa_n-1)}\left(\cdots p_n^{[2]}\boxast_n^{(2)}\left( p_{n}^{[1]}\boxast_n^{(1)} p_{X_n}\right)\cdots \right)\right).
    \end{align}
    \end{itemize} 
\end{corollary}
\begin{proof}
    \begin{itemize}
        \item[(i)] Follows from the definition of $n$-admissibility and Corollary \ref{cor:E[U^m]}.
        \item[(ii)] is obtained by applying Lemma \ref{lem:U_transf_finite_free_conv} iteratively.
        \item[(iii)] It is enough to notice that equation $X\overset{d}{=}_nAX+B$ has a unique solution as a consequence of Theorem \ref{thm:finite_free_perp} and by (ii) it is equivalent to \eqref{eq:poyperp}. 
    \end{itemize}
\end{proof}

\begin{example}
Assume that $D_1$, $D_2$ are independent. 
Simple examples of \(n\)-admissible pairs include
\begin{itemize}
    \item $(A,B) = (D_1,D_2)$. Equation \eqref{eq:finite_free_perp} is then equivalent to $p_X = p^{[2]} \boxplus_n (p^{[1]} \boxtimes_n p_{X})$.
    \item $(A,B) = (D_2,D_1 D_2)$. Equation \eqref{eq:finite_free_perp} is then equivalent to $p_X = p^{[2]} \boxtimes_n  (p^{[1]} \boxplus_n p_{X})$.
\end{itemize}
\end{example}

\begin{theorem}\label{thm:rel_roots}
 Assume $(A,B)$ is $n$-admissible such that
 \begin{itemize}
 \item[(i)] $p_n^{[i]}\in\pols_n(\R_{\geq 0})$, $i=1,\ldots,\kappa_n$,
     \item[(ii)] $| \E[A^m]| = \prod_{i\in\mathcal{I}_n} | \coef{m}{n}(p_{n}^{[i]})|<1$ for all $m=1,\ldots,n$. 
 \end{itemize}
 Then, the finite free perpetuity $p_X$ of order $n$ generated by $(A,B)$ has only real, non-negative roots. 
\end{theorem}

\begin{proof}
Let $S_N = A_N S_{N-1}+B_N$, $S_0=0$, where $(A_k,B_k)_k$ are iid copies of $(A,B)$. Then, by Corollary \ref{cor:n-admiss} (ii),
\begin{align*}
 \E[ (z-S_N)^n] &=  \E[(z-(A_N S_{N-1}+B_N))^n] \\
&= p_{n}^{[\kappa_n]}\boxast_n^{(\kappa_n)} \left( p_{n}^{[\kappa_n-1]}\boxast_n^{(\kappa_n-1)}\left(\cdots p_n^{[2]}\boxast_n^{(2)}\left( p_{n}^{[1]}\boxast_n^{(1)} p_{S_{N-1}}\right)\cdots \right)\right) (z),
\end{align*}
where $p_{S_0}(z) = e_{\boxplus_n}(z) = z^n$.
Since  $\boxtimes_n$ and $\boxplus_n$ preserve real, non-negative roots (recall Proposition \ref{prop:pres}), we show by induction that for each $N$, $p_{S_N}\in \pols_{n}(\R_{\geq 0})$.
Using Proposition \ref{prop:lim_E[S_N^k]}, we obtain $\E[X^j] = \lim_{N\to\infty}\E[S_N^j]$ for all $j\in\{1,\dots,n\}$, i.e.,
\[
p_X(z) = \lim_{N\to\infty} p_{S_N}(z).
\]
Thus, by Lemma \ref{lem:conv_rel_pols} we obtain $p_X \in \pols_{n}(\R_{\geq 0})$, which completes the proof.
\end{proof}

\section{Approximation of free perpetuities}\label{sec:approx}

In Section \ref{sec:ex} we presented an example of a finite free perpetuity, for which the empirical root distribution, as the degree converges to infinity, converges to a free perpetuity. In the following section we show that under some natural conditions, the empirical root distribution $\meas{p_{X_n}}$ converges weakly to the distribution of a free perpetuity. 

We begin by recalling the notions needed to define a free perpetuity.
Let $(\mathcal{A},\tau)$ be a tracial $W^*$-probability space, that is,
$\mathcal{A}$ is a von Neumann algebra and $\tau$ is a faithful, normal, tracial state. By $\tilde{\mathcal{A}}$ we denote the algebra of unbounded operators affiliated with $\mathcal{A}$.  Throughout, we assume that $(\mathcal A,\tau)$ is large enough to contain freely independent self-adjoint affiliated operators with any prescribed probability distributions on $\R$.

A free perpetuity is a self-adjoint operator \(\X\in \tilde{\mathcal{A}}\) satisfying the affine fixed-point equation
\begin{align}\label{eq:affine}
\X \stackrel{d}{=} \A\, \X \,\A^\ast + \B,
\qquad \X\mbox{ and }
(\A,\B)\ \text{are $*$-free},
\end{align}
where \(\A,\B\in\tilde{\mathcal{A}}\) and \(\B=\B^\ast\geq0\). In contrast with \cite{FreePerp}, we also allow \(\B=0\). 

\begin{proposition}\label{prop:unique_free}
If \(\tau(\A^\ast\A)<1\) and $\tau(\B)<\infty$,
then \eqref{eq:affine} admits a unique  solution.
\end{proposition}

\begin{proof}
By \cite[Remark 1.1]{FreePerp},  without loss of generality one may reduce \eqref{eq:affine} to the case where the coefficient multiplying \(\X\) is positive. More precisely, after replacing \((\A,\B)\) by a pair \((\M,\tilde{\B}) = (|\A|,U^\ast\B U)\), where the unitary operator $U$ from the polar decomposition of $\A=U|\A|$ and $|\A|=(\A^\ast\A)^{1/2}$ we obtain  with \(\M=\M^\ast\ge 0\), 
\begin{align}\label{eq:affine_sym}
\X \stackrel{d}{=} \M\,\X\,\M+\tilde{\B}, \qquad \X\mbox{ and } (\M,\tilde{\B})\ \text{are $*$-free}.
\end{align}

If \(\tau(\M^2)<1\), $\tau(\B)<\infty$ and $\B\geq0$, then existence follows from \cite[Theorem 4.6 (iii)]{FreePerp}. If \(\B=0\), then \(\X\stackrel{d}{=}0\) is a solution to
\eqref{eq:affine_sym}.

Uniqueness follows from \cite[Theorem 4.2]{FreePerp} (note that its proof does not use the assumption $\B\neq 0$). 
\end{proof}

We extend the notion of admissibility, introduced in Definition \ref{def:n_admissibility} for pairs of random variables, to sequences of pairs of random variables.

\begin{definition}\label{def_B_n}
A sequence $(A_n,B_n)_n$ is admissible if $(A_n,B_n)$ is $n$-admissible for each $n\in\N$ such that the numbers $\kappa:=\kappa_n\in\N$ and the sets $\mathcal{I}:=\mathcal{I}_n\subset\{1,\ldots,\kappa_n\}$ do not depend on $n$. 

\end{definition}

\begin{theorem}\label{main_thm}
Assume the sequence $(A_n,B_n)_n$ is admissible with
\begin{itemize}
\item[(i)] $p_{n}^{[i]}\in\pols_n(\R_{\geq 0})$ for all $i=1,\ldots,\kappa$, and $|\E[A_n^m]|<1$ for all $1\leq m\leq n$, 
\item[(ii)]   $\limsup_{n\to\infty}\E[A_n]<1$ and  $\limsup_{n\to\infty}\E[B_n]<\infty$,
\item[(iii)] the sequence $(\meas{p_{n}^{[i]}})_n$ converges weakly to some $\mu_i\in\MM(\R_{\geq 0})$, $i=1,\ldots,\kappa$. 
\end{itemize}
Then, for every \(n\in\N\), the finite free perpetuity \(p_{X_n}\) of order \(n\)
generated by \((A_n,B_n)\) exists, is unique, and has only nonnegative roots.

Moreover, the sequence $(\meas{p_{X_n}})_n$ converges weakly to some $\mu_{\X}\in\mathcal{M}(\R_{\ge0})$, which is uniquely determined by the free perpetuity fixed-point equation
\begin{align}\label{eq:free_perp_admis}
\mu_\X =  \mu_\kappa\boxast^{(\kappa)} \left( \mu_{\kappa-1}\boxast^{(\kappa-1)}\left(\cdots \mu_2 \boxast^{(2)}\left( \mu_1\boxast^{(1)} \mu_\X\right)\cdots \right)\right),
\end{align}
where $\boxast^{(i)} = \boxtimes$ if $i\in\mathcal{I}$ and $\boxast^{(i)} = \boxplus$ otherwise. 
\end{theorem}

\begin{proof}
We note that for all $n\in\N$ the pair $({A_n}, {B_n})$ is $n$-admissible and therefore, under (i), Theorem \ref{thm:rel_roots} implies that the finite free perpetuity $p_{X_n}$ of order $n$ generated by $(A_n,B_n)$, exists, is unique and has only real, non-negative roots. 

Denote $\nu_n = \meas{p_{X_n}}$.
First, we will show that the family $(\nu_n )_n$ is tight. Since these measures are supported in $\R_+$, it is enough to show that for any $\varepsilon>0$ there exists $K>0$ such that  
\[
\sup_{n\in\N}\nu_n\big((K, \infty)\big)<\varepsilon.
\]
Using Markov's inequality and then applying Corollary \ref{cor:E[U^m]}, we obtain that for all $t>0$,
    \begin{align}\label{mark_ineq1}
\nu_n \big((t, \infty)\big)\leq \frac{m_1(\nu_n)}{t} = \frac{\E[X_n]}{t}.
    \end{align}
   Since $X_n \stackrel{d}{=}_{n} A_n X_n+B_n$ with $X_n$ and $(A_n,B_n)$ affinely $n$-independent, we have in particular 
    \begin{align*}
\E[{X_n}]=\E[{A_n}]\E[{X_n}]+\E[{B_n}],
    \end{align*}
    which implies that (note that by (i) we have $|\E[A_n]|<1$ for each $n$)
    \begin{align*}
        \E[{X_n}]=\frac{\E[{B_n}]}{1-\E[{A_n}]}. 
    \end{align*}
    From \eqref{mark_ineq1} and assumption (ii) we conclude that there exists $C>0$ such that 
    \begin{align*}
     \nu_n \big((t, \infty)\big)\leq \frac{C}{t},
    \end{align*}
    which is less than $\varepsilon$ for sufficiently large $t$. 
  Thus, the family of measures $(\nu_n )_n$ is tight. 

 Fix a subsequence $(\nu_{n_{k}})_k$. By tightness, there exists a further subsequence and a probability measure $\nu$ such that $\nu_{n_{k_j}}\weak \nu$ as $j\to\infty$.

By the admissibility assumption, we have for each $n\in\N$, 
    \[
p_{X_n} =  p_{n}^{[\kappa]}\boxast_n^{(\kappa)} \left( p_{n}^{[\kappa-1]}\boxast_n^{(\kappa-1)}\left(\cdots p_n^{[2]}\boxast_n^{(2)}\left( p_{n}^{[1]}\boxast_n^{(1)} p_{X_n}\right)\cdots \right)\right).
    \]
    
Define the sequence of polynomials $(q_j)_j$ by
\[
q_j = p_{n_{k_j}}^{[\kappa]}\boxast_{n_{k_j}}^{(\kappa)} \left( p_{n_{k_j}}^{[\kappa-1]}\boxast_{n_{k_j}}^{(\kappa-1)}\left(\cdots p_{n_{k_j}}^{[2]}\boxast_{n_{k_j}}^{(2)}\left( p_{n_{k_j}}^{[1]}\boxast_{n_{k_j}}^{(1)} p_{X_{n_{k_j}}}\right)\cdots \right)\right).
\]
Under (iii), Corollary  \ref{cor:iterated_approximation} implies that as $j\to\infty$,
\[
\meas{q_j} \weak \mu_\kappa\boxast^{(\kappa)} \left( \mu_{\kappa-1}\boxast^{(\kappa-1)}\left(\cdots \mu_2 \boxast^{(2)}\left( \mu_1\boxast^{(1)} \nu\right)\cdots \right)\right).
\]
But, for each $j$, we have $q_j=p_{X_{n_{k_j}}}$, therefore we have also 
\[
\meas{q_j}\weak \nu.
\]
By uniqueness of the weak limit, we obtain 
\begin{align}\label{eq:nunu}
\nu =  \mu_\kappa\boxast^{(\kappa)} \left( \mu_{\kappa-1}\boxast^{(\kappa-1)}\left(\cdots \mu_2 \boxast^{(2)}\left( \mu_1\boxast^{(1)} \nu\right)\cdots \right)\right),
\end{align}

Let $\mathbb{D}_1, \dots, \mathbb{D}_\kappa \in \tilde{\mathcal{A}}$ be free, self-adjoint elements such that each $\mathbb{D}_i$ has distribution $\mu_i$. Then the right-hand side of \eqref{eq:nunu} is the distribution of the affine function of $\X\sim\nu$, free from $\mathbb{D}_{1},\dots,\mathbb{D}_{\kappa}$,
    \[
    f_\kappa\circ f_{\kappa-1}\circ\ldots f_1(\X),
    \]
    where 
    \[
    f_i(\Y) := \begin{cases}
      \mathbb{D}_i^{1/2} \Y \mathbb{D}_i^{1/2}  , & i\in \mathcal{I},\\
       \mathbb{D}_i+ \Y, & i\notin \mathcal{I}.
    \end{cases}
    \]
    Thus, $\nu$ is the distribution of $\X$ satisfying \eqref{eq:affine} for some $(\A,\B)$ depending on $(\mathbb{D}_i)_{i=1}^\kappa$. Notice that $\B$ is self-adjoint, by definition. Moreover, since $\mu_i$ are nonnegative, $\A$ and $\B$ are nonnegative.

By definition of admissibility and Corollary \ref{cor:iterated_approximation}, 
\[
\meas{p_{A_n}}=\meas{\mathop{\boxtimes_n}\limits_{i\in \mathcal I}p_n^{[i]}} \weak \mathop{\boxtimes}\limits_{i\in \mathcal I}\mu_i = :\mu_\A.
\]
Similarly,
\begin{equation*}
\begin{aligned}
\meas{p_{B_n}} 
&= \meas{ p_{n}^{[\kappa]}\boxast_n^{(\kappa)} \left( p_{n}^{[\kappa-1]}\boxast_n^{(\kappa-1)}\left(\cdots p_n^{[2]}\boxast_n^{(2)}\left( p_{n}^{[1]}\boxast_n^{(1)} e_{\boxplus_n}\right)\cdots \right)\right) } \\[1ex]
&\weak \mu_\kappa\boxast^{(\kappa)} \left( \mu_{\kappa-1}\boxast^{(\kappa-1)}\left(\cdots \mu_2 \boxast^{(2)}\left( \mu_1\boxast^{(1)} \delta_0\right)\cdots \right)\right)=:\mu_\B.
\end{aligned}
\end{equation*}

Since all measures are supported on $\R_{\geq 0}$, by the Portmanteau theorem and Corollary \ref{cor:E[U^m]}, we have
\begin{align*}
\tau(\A^\ast\A)= m_1(\mu_\A) \leq \liminf_{n\to\infty} m_1(\meas{p_{A_n}}) \leq \limsup_{n\to\infty} \E[A_n]<1, 
\intertext{and}
\tau(\B)= m_1(\mu_\B) \leq \liminf_{n\to\infty} m_1(\meas{p_{B_n}}) \leq \limsup_{n\to\infty} \E[B_n]<\infty. 
\end{align*}

By Proposition \ref{prop:unique_free}, the solution to \eqref{eq:nunu} is unique. Thus, the measure $\nu$ agrees on every convergent subsequence. Therefore $(\nu_n)_n$ converges weakly to $\mu_\X:=\nu$.
    
\end{proof}

Using Corollary \ref{cor:n-admiss} (i) and (iii), we restate Theorem \ref{main_thm} purely in terms of polynomials. Recall that $G_{\mathcal{I},\kappa}$ is defined before Definition \ref{def:n_admissibility}. 
\begin{theorem}[Polynomial reformulation of Theorem~\ref{main_thm}]\label{corr:reform}
    Fix $\kappa\in\N$ and $\mathcal{I}\subset\{1,\dots,\kappa\}$. Let $(p_n^{[j]})_{n \in \N, 1 \le j \le \kappa}$ be a family of polynomials such that $p_n^{[j]} \in \mathcal{P}_n(\R_{\ge 0})$ for all $1\le j\le\kappa$. Assume that 
    \begin{itemize}
        \item[(i)] $\prod_{i\in\mathcal{I}}|\coef{m}{n}(p_n^{[i]})|<1$ for all $1\le m\le n$,
        \item[(ii)]$\limsup_{n\to\infty}\prod_{i\in\mathcal{I}}\coef{1}{n}(p_n^{[i]})<1$ and $\limsup_{n\to\infty}G_{\mathcal{I},\kappa}(0, \coef{1}{n}(p_n^{[1]}),\dots,\coef{1}{n}(p_n^{[\kappa]}))<\infty$,
        \item[(iii)] the sequence $(\meas{p_{n}^{[i]}})_n$ converges weakly to some $\mu_i\in\MM(\R_{\geq 0})$, $i=1,\ldots,\kappa$. 
    \end{itemize}
    Then, for all $n\in\N$, there exists a unique $p_{X_n}\in\pols_n(\C)$ satisfying \eqref{eq:poyperp} with $\kappa_n=\kappa$ and $p_{X_n}$ has only real, non-negative roots. 
    
    Moreover, the sequence $(\meas{p_{X_n}})_n$ converges weakly to some $\mu_{\X}\in\mathcal{M}(\R_{\ge0})$, which is uniquely determined by the free perpetuity fixed-point equation \eqref{eq:free_perp_admis}.
\end{theorem}

\begin{remark}
    Fix $\kappa\in\N$ and $\mathcal{I}\subset\{1,\dots,\kappa\}$. Let $(p_n^{[j]})_{n \in \N, 1 \le j \le \kappa}$ be a family of polynomials such that $p_n^{[j]} \in \mathcal{P}_n(\R_{\ge 0})$ for all $1\le j\le\kappa$. Assume $p_n^{[i]}$ is $d$-monic with $d\in\N$ for all $i\in\mathcal{I}$ and sufficiently large $n$. Then, assumption (i) in Corollary \ref{corr:reform} is clearly not satisfied. However, the theorem may still be applied to 
    \[
q_{n-d}^{[j]} := 
\begin{cases} 
    \partial_n^dp_n^{[j]}, & \text{if } j\notin\mathcal{I}, \\
    D_n^dp_n^{[j]},   & \text{if } j\in\mathcal{I}.
\end{cases}
\]
By Proposition \ref{prop:derivatives_preserve_nonnegative_roots},  $q_{n-d}^{[j]}\in\pols_{n-d}(\R_{\ge 0})$ for all $j\in\{1,\dots,\kappa\}$. By Lemma \ref{lemma:normalized_derivatives_coefficients}, we also have the following relation of coefficients of $q_{n-d}^{[j]}$ to coefficients of $p_{n}^{[j]}$: for $m\in\{0,\dots,n-d\}$
\[
\coef{m}{n-d}(q_{n-d}^{[j]}) = 
\begin{cases} 
    \coef{m}{n}(p_{n}^{[j]}), & \text{if } j\notin\mathcal{I}, \\
    \coef{m+d}{n}(p_{n}^{[j]}),   & \text{if } j\in\mathcal{I}.
\end{cases}
\]
We note that $p_{X_n}\in\pols_{n-d}(\C)$ is a solution to \eqref{eq:poyperp} with $\kappa_n=\kappa$ if and only if $p_{X_n}$ satisfies
\[
p_{X_n} =  q_{n-d}^{[\kappa]}\boxast_{n-d}^{(\kappa)} \left( q_{n-d}^{[\kappa-1]}\boxast_{n-d}^{(\kappa-1)}\left(\cdots q_{n-d}^{[2]}\boxast_{n-d}^{(2)}\left( q_{n-d}^{[1]}\boxast_{n-d}^{(1)} p_{X_n}\right)\cdots \right)\right).
\]
\end{remark}

\appendix

\section{Complex-valued perpetuities}\label{app}

Assume that $A,B,X$ are complex-valued random variables and consider the affine fixed-point equation of the form
\begin{equation}\label{eq:classical-perpetuity_app}
X \stackrel{d}{=} AX+B,\qquad X\mbox{ and }(A,B)\mbox{ are independent}.
\end{equation}
A random variable $X$ satisfying
\eqref{eq:classical-perpetuity_app} is called a complex-valued perpetuity.

In the literature, perpetuities are usually studied for real $A,B,X$; see \cite{BurBook16}. The case of complex-valued random variables $A,B,X$ can be rewritten equivalently as a real-vector-valued perpetuity, namely
\begin{align}\label{eq:vecperp}
v(X) \stackrel{d}{=} m(A) v(X)+v(B),\qquad v(X)\mbox{ and }(m(A), v(B))\mbox{ are independent},
\end{align}
where 
\[
v(x+i y) = \begin{bmatrix} x \\ y \end{bmatrix}\quad\mbox{and}\quad m(x+i y) = 
 \begin{bmatrix}
x & -y \\ y & x
 \end{bmatrix},\quad x,y\in\R.
 \]
We equip $\R^2$ with the Euclidean norm and use the corresponding induced operator norm for real \(2\times2\) matrices (denoted by $\|\cdot\|$). Then, for every $z=x+iy\in\C$,
\[
|v(z)|=|z|
\qquad\mbox{and}\qquad
\|m(z)\|=|z|.
\]
 
Let $(A_n,B_n)_{n\ge1}$ be iid copies of $(A,B)$ and set
\[
\Pi_0:=I_2,
\qquad
\Pi_n:=m(A_1)\cdots m(A_n),\quad n\ge1.
\]
The canonical candidate for a solution is the random series
\begin{equation}\label{eq:classical-series}
v(X) \stackrel{d}{=} \sum_{n=1}^\infty \Pi_{n-1}v(B_n).
\end{equation}
A standard sufficient condition for almost sure convergence of \eqref{eq:classical-series}, and for existence and uniqueness in law of a solution to \eqref{eq:vecperp}, is  (see e.g. \cite[Theorem 4.1.4]{BurBook16})
\begin{equation}\label{eq:classical-existence}
\E[\log^+ \| m(A)\|]<\infty,
\qquad
\E[\log^+ |v(B)|]<\infty,\qquad \gamma<0,
\end{equation}
where $\gamma$ is the top Lyapunov exponent 
\[
\gamma = \inf_{n\in\N} \frac1n\E[\log \|\Pi_n\|]
\]
and $\log^+ x:=\max\{\log x,0\}$.  

Noting that $m$ is a real $2\times 2$ matrix representation of complex multiplication, we have
\[
\Pi_n = m(A_1 \cdots A_n).
\]
Thus,  $\|\Pi_n\| = |A_1| \cdots|A_n|$ and therefore 
\[
\gamma=\E[\log |A|].
\]
Clearly, if $\E[\log|A|]<0$, then $\E[\log^+ \| m(A)\|] = \E[\log^+|A|]<\infty$. 

Thus, the conditions  
\[
\E[\log |A|]<0
\qquad\mbox{and}\qquad
\E[\log^+|B|]<\infty
\]
imply almost sure convergence of the series 
\[
\sum_{n=1}^\infty A_1\cdots A_{n-1}B_n.
\]
The law of this series is the unique solution in law to 
\eqref{eq:classical-perpetuity_app}.

\section*{Acknowledgements}
The authors would like to thank Jonas Jalowy for helpful comments that led to a generalization of Theorem~\ref{main_thm}.

\bibliographystyle{plain}
\bibliography{Bibl}

\end{document}